\documentclass{amsart}
\usepackage{amssymb,latexsym,color}
\theoremstyle{plain}
\newtheorem{theorem}{Theorem}
\newtheorem{corollary}{Corollary}
\newtheorem{proposition}{Proposition}
\newtheorem{lemma}{Lemma}
\theoremstyle{definition}
\newtheorem{definition}{Definition}

\newtheorem{notation}{Notation}
\newtheorem{remark}{Remark}
\newtheorem{question}{Question}

\newcommand{\red}[1]{{\color{black}#1}}

\begin{document}

\title[Symmetric border rank  4]{Stratification of the fourth secant variety of Veronese varieties via the symmetric rank}

\author{Edoardo Ballico, Alessandra Bernardi}
\address{Dept. of Mathematics\\
  University of Trento\\
38123 Povo (TN), Italy}
\address{University of Turin, Department of Mathematics ``Giuseppe Peano'', Via Carlo Alberto 10, I-10123 (TO), 
 Italy.}
\email{ballico@science.unitn.it, alessandra.bernardi@unito.it}
\thanks{The authors were partially supported by CIRM of FBK Trento 
(Italy), Project Galaad of INRIA Sophia Antipolis M\'editerran\'ee 
(France), Institut Mittag-Leffler (Sweden), Marie Curie: Promoting science (FP7-PEOPLE-2009-IEF), MIUR and GNSAGA of 
INdAM (Italy).}

\begin{abstract}
If $X\subset \mathbb{P}^n$  is a projective non degenerate variety, the $X$-rank of a point $P\in \mathbb{P}^n$ is defined to be the minimum integer $r$ such that $P$ belongs to the span of $r$ points of $X$. We describe the complete stratification of the fourth secant variety of any Veronese variety $X$ via the $X$-rank. This result has an equivalent translation in terms either of symmetric tensors or of homogeneous polynomials. It allows to classify all the possible integers $r$ that can occur in the minimal  decomposition of a homogeneous polynomial of $X$-border rank $4$ (i.e. contained in the fourth secant variety) as a linear combination of  powers of linear forms.
\end{abstract}

\maketitle

\section*{Introduction}

Fix integers $m \ge 2$ and $d\ge 2$ and set $n_{m,d}:= \binom{m+d}{m}-1$. All along this paper the field $\mathbb {K}$ over which all the projective spaces and all the vector spaces will be defined is algebraically closed and of characteristic $0$.
Let $\nu_{m,d} : \mathbb{P}^m \hookrightarrow \mathbb {P}^{n_{m,d}}$ 
be the order $d$ Veronese embedding of $\mathbb {P}^m$ defined by the sections of the sheaf $\mathcal{O}_{\mathbb {P}^m}(d)$. Set: 
\begin{equation}\label{Xmd}
X_{m,d}:= \nu_{m,d}(\mathbb{P}^m). 
\end{equation}
We often set $X:= X_{m,d}$ and $n:= n_{m,d}$. 
The Veronese variety can be regarded both as the variety that parameterizes projective classes of homogeneous polynomials of degree $d$ in $m+1$ variables that can be written as $d$-th powers of linear forms, and as the variety that parameterizes projective classes of symmetric tensors $T\in V^{\otimes d}$ where $V$ is a vector space of dimension $m+1$ and $T=v^{\otimes d}$ for certain $v\in V$ (symmetric tensors of the form $v^{\otimes d}$ are often called ``completely decomposable tensors''). Hence if we indicate with $\mathbb {K}[x_0, \ldots , x_m]_d$ the vector space of homogeneous polynomials of degree $d$ in $m+1$ variables, and with $S^dV$ the subspace of symmetric tensors in $V^{\otimes d}$, then the Veronese variety $X_{m,d}\subset \mathbb{P}^{n_{m,d}}$ can be described both as $\{ [F]\in \mathbb{P}^{}(K[x_0, \ldots , x_m]_d) \, | \, \exists \, L\in K[x_0, \ldots , x_m]_1 \, \mathrm{ s.t. } \, F=L^d \}$ and as $\{[T]\in \mathbb{P}^{}(S^dV)\, | \, \exists \, v\in V \, \mathrm{ s.t. } \, T=v^{\otimes d}\}$.

A very classical problem coming from a number theory problem known as the Big Waring Problem, consists in determining the minimum integer $s$ for which a generic form $F\in K[x_0, \ldots , x_m]_d$  can be written as the sum of $s$ $d$-th powers of linear forms $L_1, \ldots , L_s \in K[x_0, \ldots , x_m]_1$:
\begin{equation}\label{Li}
F=L_1^d+ \cdots + L_s^d.
\end{equation}
 The same $s$ gives  the minimum integer for which the generic symmetric tensor $T\in S^{d}V$ can be written as a sum of $s$ completely decomposable tensors $v_1^{\otimes d}, \ldots , v_s^{\otimes d}\in S^dV$:
\begin{equation}\label{vi}
T=v_1^{\otimes d}+ \cdots + v_s^{\otimes d}.
\end{equation}
This problem was solved by J. Alexander and A. Hirschowitz  in \cite{ah1} (see also \cite{bo} for a modern proof).

Let $Y\subset \mathbb {P}^r$ be any integral and non-degenerate variety. Fix $P\in \mathbb {P}^r$. The $Y$-rank $r_Y(P)$ of $P$ is the minimal cardinality of a finite set $S\subset Y$ such that $P\in \langle S\rangle$. 

Let $\sigma_s(Y)\subset \mathbb{P}^n$ be the so called ``~higher $s$-th secant variety of $Y$~" (for brevity we will quote it only as the ``~$s$-th secant variety of $Y$~"):
\begin{equation}\label{sigma}
\sigma_s(Y):=\overline{\bigcup_{P_1, \ldots , P_s\in Y}\langle P_1, \ldots , P_s \rangle}.
\end{equation}
From this definition it turns out that a generic element of $\sigma_s(Y)$ has $Y$-rank equal to $s$, but obviously not all the elements of $\sigma_s(Y)$ have $Y$-rank equal to $s$ (except in some cases, like $s=1$ or $r= \dim (Y)+1$). 
For any $P\in \mathbb {P}^r$ the $Y$-border rank $b_Y(P)$ of $P$ is the minimal integer $s$ such that $P\in \sigma _s(Y)$. We have $b_Y(P) \le r_Y(P)$ and equality holds for a general point of each $\sigma _s(Y)$.
If $Y=X_{m,d}$ and $r = n_{m,d}$, then the integer
$r_{X_{m,d}}(P)$ is usually called the {\emph {symmetric tensor rank}} or {\emph {symmetric rank}} of $P$, while sometimes $b_{X_{m,d}}(P)$ is called either ``the secant rank of $P$" or the ``symmetric border rank of $P$''.

A natural question arising from the applications (see for example \cite{ACCF}, \cite{DM}, \cite{Ch}, \cite{dLC}, \cite{McC}, \cite{Co1}, \cite{JS}, \cite{l}) is: 
\begin{question}\label{Q}
Given a symmetric tensor $T\in S^{d}V$ (or a homogeneous polynomial $F\in \mathbb {K}[x_0, \ldots , x_m]_d$), which is the minimum integer $r$ for which we can write it as a linear combination  of $r$ completely decomposable tensors, i.e. as in (\ref{vi}) with $r=s$ (or as a linear combination  of $r$ $d$-th powers of linear forms, i.e. as in (\ref{Li}) with $r=s$)? 
\end{question}

Obviously answering to Question \ref{Q} for $T$ or $F$ is equivalent to finding the $X$-rank of the associated $P\in \mathbb{P}^n$, $n:= n_{m,d}$. Since $\mathbb{P}^n$ can be viewed either as the projective space of homogeneous polynomials of degree $d$ in $m+1$ variables or as the projective space of symmetric tensors of order $d$ over an $m+1$ dimensional vector space, an answer to Question \ref{Q} for a given $T$ or $F$ (resp. for all $T$ and $F$) is equivalent to compute $r_X(P)$ for the point $P\in \mathbb {P}^n$ associated to $T$ or $F$
(resp. for all $P\in \mathbb {P}^n$).

The answer to Question \ref{Q}, in the particular case of $m=1$, is known since Sylvester (\cite{cs}, \cite{lt}, Theorem 4.1, \cite{bgi}, \S 3). In that case the Veronese variety coincides with a rational normal curve. In \cite{bcmt} and \cite{cglm} one can find an algebraic theoretical algorithm for the general case with $m\geq 2$.

Both the Big Waring Problem and  Question \ref{Q} have a very interesting reformulation in Algebraic Geometry by using Linear Algebra tools. The authors of \cite{bgi} give some effective algorithms for the computation of the $X$-rank of certain kind of symmetric tensors by using this algebraic geometric interpretation. The advantage of those last algorithms is that they are effective and that they arise from an algebraic geometric perspective that gives the idea on how one can proceed in the study of the $X$-rank either of a form or of a symmetric tensor. Let us go into the details of that geometric description.

First of all, the definition (\ref{sigma}) of the secant varieties of the Veronese variety implies the following chain of containments:
\begin{equation}\label{chain}
X=\sigma_1(X)\subseteq \sigma_2(X)\subseteq \cdots \ \subseteq \sigma_{k-1}(X)\subseteq \sigma_{k}(X)=\mathbb{P}^n
\end{equation}
for certain natural number $k$. Therefore $\sigma_s(X)$ contains all the elements of $X$-rank less or equal than $s$. 
\\
Moreover the set
\begin{equation}\label{sigma0}
\sigma_s^0(X):=\bigcup_{P_1, \ldots ,P_s\in X}\langle P_1, \ldots , P_s \rangle
\end{equation}
is contained in $\sigma_s(X)$ and it is made by the elements $P\in \mathbb{P}^n$ whose $X$-rank is less or equal than $s$, hence the elements of $\sigma_s(X)\setminus (\sigma_{s-1}(X) \cup \sigma_s^0(X))$ have $X$-rank bigger than $s$.

What is done in \cite{bgi} is to start giving a stratification of $\sigma _s(X)\setminus \sigma _{s-1}(X)$ via the $X$-rank: in that paper the cases of $\sigma_2(X_{m,d})$ and $\sigma_3(X_{m,d})$ for any $m,d\geq 2$  are completely classified (among others). The authors give algorithms that produce the $X$-rank of an element of $\sigma_2(X_{m,d})$ and $\sigma_3(X_{m,d})$. 
\\
If we indicate
\begin{equation}\label{sigmasr}
\sigma_{s,r}(X):=\{P\in \sigma_s(X)\setminus \sigma _{s-1}(X) \, | \, r_X(P)=r\}\subset \sigma_s(X)\setminus \sigma _{s-1}(X)\subset \mathbb{P}^n,
\end{equation}
 then we can write the stratifications quoted above as follows:
\begin{itemize}
\item $\sigma_2(X_{m,d})\setminus X_{m,d}=\sigma_{2,2}(X_{m,d})\cup \sigma_{2,d}(X_{m,d}), \hbox{ for }m\geq 1 \hbox{ and } d\geq 2$
(cfr. \cite{cs}, \cite{bgi}, \cite{cglm}, \cite{bcmt});
\item $\sigma_3(X_{1,d})\setminus \sigma_2(X_{1,d})=\sigma_{3,3}(X_{1,d})\cup \sigma_{3,d-1}(X_{1,d}), \hbox{ for } d\geq 4$ (cfr. \cite{cs}, \cite{bgi}, \cite{cglm}, \cite{bcmt});
\item $\sigma_3(X_{m,3})\setminus \sigma_2(X_{m,3})=\sigma_{3,3}(X_{m,3})\cup \sigma_{3,4}(X_{m,3})\cup \sigma_{3,5}(X_{m,3}), \hbox{ for } m\geq 2 $ (see \cite[Theorem 40]{bgi});
\item $\sigma_3(X_{m,d})\setminus \sigma_2(X_{m,d})=\sigma_{3,3}(X_{m,d})\cup \sigma_{3,d-1}(X_{m,d})\cup \sigma_{3,d+1}(X_{m,d})\cup \sigma_{3,2d-1}(X_{m,d})$,  for  $m\geq 2$  and $d\geq 4$ 
(see \cite[\S 4]{bgi} or \cite{kl} for the cases $d=3,4$).
\end{itemize}

What we want to do in this paper is to give the analogous stratification for $\sigma_4(X_{m,d})$ for any $m,d\geq 2$. We will prove the following:

\begin{theorem}\label{theorem} The stratification of $\sigma_4(X_{m,d})\setminus \sigma_3(X_{m,d})$ via the $X_{m,d}$-rank is the following:
\begin{itemize}
\item[(a)] $\sigma_4(X_{1,d})\setminus \sigma_3(X_{1,d})=\sigma_{4,4}(X_{1,d})\cup \sigma_{4,d-2}(X_{1,d})$, if $d\geq 6$;
\item[(b)] $\sigma_4(X_{2,3})\setminus \sigma_3(X_{2,3})=\mathbb{P}^9\setminus \sigma_3(X_{2,3})=\sigma_{4,4}(X_{2,3});$
\item[(c)] $\sigma_4(X_{2,4})\setminus \sigma_3(X_{2,4})=\sigma_{4,4}(X_{2,4})\cup \sigma_{4,6}(X_{2,4})\cup \sigma_{4,7}(X_{2,4});$
\item[(d)] $\sigma_4(X_{2,5})\setminus \sigma_3(X_{2,5})=\sigma_{4,4}(X_{2,5})\cup\sigma_{4,5}(X_{2,5})\cup \sigma_{4,7}(X_{2,5})\cup \sigma_{4,8}(X_{2,5})\cup \sigma_{4,9}(X_{2,5});$
\item[(e)] $\sigma_4(X_{2,d})\setminus \sigma_3(X_{2,d})=\sigma_{4,4}(X_{2,d})\cup \sigma_{4,d-2}(X_{2,d})\cup\sigma_{4,d}(X_{2,d})\cup \sigma_{4,d+2}(X_{2,d})\cup \sigma_{4,2d-2}(X_{2,d})$,
%\cup \sigma_{4,2d-1}(X_{2,d})$
if $d\geq 6$;
\red{
\item[(f)] $\sigma_4(X_{m,3})\setminus \sigma_3(X_{m,3})=\sigma_{4,4}(X_{m,3})\cup \sigma_{4,5}(X_{m,3})\cup\sigma_{4,6}(X_{m,3})\cup \sigma_{4,7}(X_{m,3})$, if $m\geq 3$;
}
\item[(g)] $\sigma_4(X_{m,4})\setminus \sigma_3(X_{m,4})=\sigma_{4,4}(X_{m,4})\cup\sigma_{4,6}(X_{m,4})\cup %\sigma_{4,7}(X_{m,4})\cup 
\sigma_{4,8}(X_{m,4})\cup \sigma_{4,10}(X_{m,4})$, if $m\geq 3$; 
\item[(h)] $\sigma_4(X_{m,5})\setminus \sigma_3(X_{m,5})=\sigma_{4,4}(X_{m,5})\cup \sigma_{4,5}(X_{m,5})\cup\sigma_{4,7}(X_{m,5})\cup\sigma_{4,8}(X_{m,5})\cup% \sigma_{4,9}(X_{m,5})\cup
\sigma_{4,10}(X_{m,5})\cup \sigma_{4,13}(X_{m,5})$, if $m\geq 3$;
\item[(i)] $\sigma_4(X_{m,d})\setminus \sigma_3(X_{m,d})=\sigma_{4,4}(X_{m,d})\cup \sigma_{4,d-2}(X_{m,d})\cup\sigma_{4,d}(X_{m,d})\cup\sigma_{4,d+2}(X_{m,d})\cup\sigma_{4,2d-2}(X_{m,d})\cup \sigma_{4,2d}(X_{m,d})\cup \sigma_{4,3d-2}(X_{m,d})$, if $m\geq 3$ and $d\geq 6$.
\end{itemize}
Moreover all listed $\sigma _{s,r}(X_{n,d})$ are non-empty.
\end{theorem}

The case of the rational normal curve, i.e.  item (a), is due to Sylvester and done in \cite{cs}, \cite{cglm}, \cite{bcmt}, \cite[Theorem 4.1]{lt},  and \cite[\S 3]{bgi}.
\\
The cases of the Veronese surfaces in degrees $3$ and $4$, i.e. items (b) and (c), are done in \cite[Theorems 40, 44]{bgi} respectively and in \cite{kl}.
\\
We complete the case of the Veronese surfaces (items (d) and (e)) in the Subsection \ref{m2}.
\\
In the Subsection \ref{m3} we will give the stratification of $\sigma_4(X_{3,d})$, with $d\geq 3$, that will be the same stratification for any $m\geq 3$ (items (f), (g), (h) and (i)).

\begin{remark} During the Spring Semester 2001 ``Algebraic Geometry with a view towards applications" at the Institut Mittag-Leffler (Sweeden), J. H. Rhodes stated the following:

\textbf{Conjecture 1 (J. H. Rhodes):} \emph{The maximal $X_{m,3}$-rank of a homogeneous polynomial of degree 3 in $m+1$ variables   belonging to $\sigma_{m+1}(X_{m,3})$ is 2(m+1)-1.} 

We remark that the original version of this conjecture was about tensors not necessarily symmetric. For sake of completeness we quote the first version of it but we don't enter into details. 

\textbf{Conjecture 0 (J. H. Rhodes):} \emph{The maximal rank of an $(m+1) \times (m+1) \times (m+1)$ tensor of border rank $m+1$ is $2(m+1)-1$.}

Conjecture 1 is known to be true for $m=1, 2$ ($m=1$ is classical from Sylvester, and $m=2$ is \cite[Theorem 40]{bgi}).
We like to observe here that the case (f) of our Theorem \ref{theorem} shows that  this conjecture is true  also for $m=3$. Moreover from the proof of our Theorem \ref{theorem}, it is also possible to describe the structure of the schemes that evince the $X_{m,3}$-border rank of the degree 3 polynomials of maximal $X_{m,3}$-rank (they are all described by Proposition \ref{3bp}). Finally, putting $d=3$ and $m=3$, one can use our Remark \ref{3b} to produce algorithmically several sets of points evincing $r_{X_{3,3}}(P)$ for points $P\in \sigma_{4}(X_{3,3})$ of maximal $X_{3,3}$-rank.
\end{remark}

Before going into the details of the proof we need some preliminary and auxiliary sections. In Section \ref{S1} we present  the construction that will allow to associate two different 0-dimensional schemes of $\mathbb{P}^m$ to two 0-dimensional sub-schemes of $X_{m,d}$ realizing the $X_{m,d}$-border rank and the $X_{m,d}$-rank of a point $P\in \mathbb{P}^{n_{m,d}}$. We also discuss for which 0-dimensional schemes $A \subset \mathbb {P}^m$ of degree $4$ a general point of $\langle \nu _{m,d}(A)\rangle$
belongs to $\sigma _4(X_{m,d})\setminus \sigma _3(X_{m,d})$.
In Section \ref{S2} we give bounds for the $Y$-rank of a point with respect to some particular projective curves $Y\subset \mathbb{P}^t$ that will be used in the proof of the Theorem \ref{theorem}. Section \ref{S3} is made by preliminary lemmas on the linear dependence of the pre-image via the Veronese map $\nu_{m,d}$ of the 0-dimensional schemes realizing the $X$-rank and the $X$-border rank of a point $P\in \mathbb{P}^{n}$. Finally in Section \ref{S4} we collect all the previous results into the proof of Theorem \ref{theorem}.

Moreover we will describe case by case how to find the scheme that realizes the $X$-rank of a point $P$ (modulo the scheme that realizes the $X$-border rank). This allows us to give many informations on the subset $\sigma_{s,r}(X)\subset \sigma_{s}(X)$ defined in (\ref{sigmasr})
and to construct all $P\in \sigma _{4,s}(X)$ by playing with certain 0-dimensional schemes.

We like to stress here that the defining ideals of $\sigma_2(X_{1,d})$ and $\sigma_3(X_{2,d})$ are known (see \cite{k} and \cite{o} respectively) and this allows the authors of \cite{bgi} to give algorithms for the $X$-rank of points in $\sigma_s(X)$ with $s=2,3$. Given an element $P\in \mathbb{P}^n$ they can firstly check if its $X$-border rank is actually either $2$ or $3$, and then they can produce the algorithm for the computation of the $X$-rank of $P$. Unfortunately, at our knowledge, equations defining $\sigma_4(X_{3,d})$ at least set-theoretically are not known yet, therefore we could write algorithms for the $X$-rank of an element  $P\in\sigma_4(X)$ only if we already knew in some way that $b_X(P)=4$. 

\section{Preliminaries}\label{S1}

In this paper we want to study the $X$-rank of the points $P$ belonging to the fourth secant variety of the Veronese variety $X$, i.e. $P\in\sigma_4(X)$. By the chain of containments (\ref{chain}) we have that $\sigma_3(X)\subseteq \sigma_4(X)$. Since the stratification of $\sigma_3(X)$ via the $X$-rank is already known by \cite{bgi}, it is sufficient to understand the $X$-rank of points $P\in \sigma_4(X)\setminus \sigma_3(X)$.

Moreover the definition (\ref{sigma0}) of $\sigma_s^0(X)$ implies that if $P\in \sigma_4^0(X)$ then $r_X(P)\leq 4$, hence, for the purpose of this paper, it is sufficient  to study the $X$-rank of points belonging to $\sigma_4(X)\setminus (\sigma_3(X)\cup \sigma_4^0(X))$.
Before starting our construction by taking $P\in\sigma_4(X)$ we introduce the following
Remark \ref{Z++} that, for such a point $P$, gives the existence of a 0-dimensional
scheme $Z \subset X$ of degree $4$ such that $P\in \langle Z \rangle$ and $P\notin\langle Z'\rangle$ for any $Z'\subset X$ with $\deg(Z')<\deg(Z)$.

\begin{remark}\label{Z++}
Fix integers $m \ge 1$, $d \ge 2$ and $P\in \mathbb {P}^n$ such that $b_X(P) \le d+1$. By \cite{bgl},
Lemma 2.1.5 and Lemma 2.4.4, there is a smoothable 
0-dimensional and Gorenstein scheme $\mathcal{E} \subset X_{m,d}$ such that $\deg (\mathcal{E}) =
b_X(P)$,  $P\in \langle \mathcal{E}\rangle$ and $P\notin \langle \mathcal{E}'\rangle$ for any $\mathcal {E}'\subsetneq \mathcal {E}$.
\end{remark}

Before entering into the details of our construction we need to distinguish the Gorenstein cases from the non Gorenstein ones that we won't have to treat thanks to \cite{bgl}.

\begin{lemma}\label{oo1}
Let $U$ be smooth quasi-projective surface and fix $O\in U$. Let $2O$ be the 0-dimensional subscheme of $U$ with $(\mathcal {I}_{O,U})^2$ as its ideal sheaf. Let $E\subset U$ be another  0-dimensional scheme supported at $O$ with  $\deg (E) = 4$. Assume that $E$ is Gorenstein, but not curvilinear.
The latter condition is equivalent to requiring $2O \subseteq E$. Let $\mathbb {K}[[x,y]]$ be the completion of $\mathcal {O}_{U,O}$ and let $J\subset \mathbb {K}[[x,y]]$ be an ideal such
that $\mathcal {O}_E \cong \mathbb {K}[[x,y]]/J$. Then $J$ is generated by $L_1^2$ and $L_2^2$ with $L_1, L_2\in \mathbb {K}[x,y]_1$.
\end{lemma}

\begin{proof}
Since $\dim (U) =2$, $U$ is smooth and $E_{red} = \{O\}$, $E$ is not curvilinear if and only if $2O \subseteq E$. Since $\mathbb {K}[[x,y]]$ is a regular local ring and $J$ is a codimension $2$ Gorenstein ideal,
$J$ is a complete intersection (\cite[Corollary 21.20]{e}). Take generators $f_1, f_2$ of $J$ and call $Q_i(x,y)$ the first non-zero term of the Taylor expansion of $f_i$. Set $\mu _i:= \deg (Q_i(x,y))$.
Since $E$ is not curvilinear, we have $\mu _i \ge 2$ for all $i$. Since $\deg (E) = 4 = 2\cdot 2$, we have $\mu _1 = \mu _2 =2$ and $Q_1, Q_2$ have no common factor (Property (5) of \cite[\S 3.3]{f1}).
See $x, y$ as homogeneous coordinates of $\mathbb {P}^1$. The pencil $\lambda Q_1+\mu Q_2$, $(\lambda ,\mu )\in \mathbb {K}^2\setminus \{(0,0)\}$, has at least one singular element. Hence, up to a linear change of coordinates,
we may assume $Q_1 = x^2$ and write $Q_2 =  axy +by^2$. Since $Q_1$ and $Q_2$ have no common factor, we have $b\ne 0$. Set $L_1:= x$ and $L_2:= y+ax/2b$. Since $(Q_1, Q_2) = (L_1^2,L_2^2)  = (L_1^2,L_2^2,x^3,x^2y,xy^2,y^3) = (Q_1,Q_2,x^3,x^2y,xy^2,y^3) \supseteq J$
and $\mathbb {K}[[x,y]]/(L_1^2,L_2^2)$ is a $4$-dimensional $\mathbb {K}$-vector space,
we get $J = (L_1^2,L_2^2)$.\end{proof}

\begin{remark}\label{ee4} As already observed, thanks to \cite[Lemma 2.4.4]{bgl},
in the list of all cases potentially appearing in Theorem \ref{theorem} we will
only need to check the Gorenstein schemes $A \subset \mathbb {P}^m$ with $\deg (A) = 4$. The scheme $A$ is Gorenstein if and only if every connected
component of $A$ is Gorenstein. Let $E$ be a connected 0-dimensional scheme of degree $s\le 4$. If $s\le 3$, then it is Gorenstein if and only if it is curvilinear and it is one
of the ones appearing in \cite[\S 4]{bgi} in the description of $\sigma _s(X)$. Now assume $s=4$. Let $e$ be the embedding dimension of $E$, i.e. set $ e:= \dim _{\mathbb {K}}({\bf {m}}_E/{\bf {m}}_E^2)$, where
${\bf {m}}_E$ is the maximal ideal of the local ring $\mathcal {O}_E$. Since $\mathcal {O}_E/{\bf {m}}_E = \mathbb {K}$, we have $e \le \deg (E) -1 = 3$ and equality holds if and only if ${\bf {m}}_E^2=0$, i.e. if and only
if $E$ is isomorphic to the scheme $\mathbb {K}[x,y,z]/(x,y,z)^2$. Since the latter scheme is not Gorenstein, it is sufficient to look at the case $e\le 2$. If $e=1$, then $E$ is Gorenstein, because every effective divisor of a smooth
curve is a locally complete intersection and any local complete intersection is Gorenstein (\cite[Corollary 21.29]{e}). The case $e=2$ is described in Lemma \ref{oo1} .

If $d \ge 7$ the scheme evincing the border rank is unique, because the union of two such schemes has degree at most $8$ and we may apply \cite[Lemma 1]{bb} (stated in Section \ref{S2} as Lemma \ref{v1}) to these schemes instead of $A$ and $B$. Hence each $P\in \sigma _4(X)$ is associated to a unique $Z$ and $\langle \nu _{m,d}(A_1)\rangle \cap \langle \nu _{m,d}(A_2)\rangle =\langle \nu _{m,d}(A_1\cap A_2)\rangle$
for all $A_i \subset \mathbb {P}^m$ with $\deg (A_i)\le 4$, $i=1,2$. Only if $d\ge 7$ we claim that a general $P\in \langle Z\rangle$ belongs $\sigma _4(X)\setminus \sigma _3(X)$.
\end{remark}

\begin{definition}\label{Z}
Let $P\in \sigma _4(X)\setminus \sigma _3(X)$; we  say that a 0-dimensional scheme $Z\subset X$ such that $\deg (Z) = 4$ and $P\in \langle Z \rangle$ ``~evinces the $X$-border rank of $P$~". 
\end{definition}

\begin{notation}\label{Z} If $P\in \sigma_4(X)\setminus (\sigma_3(X)\cup \sigma_4^0(X))$, we fix $Z\subset X$ to be one of the degree $4$ non-reduced 0-dimensional schemes that evince the $X$-border rank of $P$, i.e. $P\in \langle Z \rangle$ and $P\notin \langle Z' \rangle$ for all 0-dimensional schemes $Z'\subset X$ of degree less or equal than $3$. If $Z = \nu _{m,d}(A)$, then we also say (with an abuse of notation) that $A$ evinces the $X_{m,d}$-border rank of $P$.
\end{notation}

In order to study the stratification of $\sigma_4(X_{m,d})$ it is therefore necessary to understand the $X_{m,d}$-rank of the points belonging to the span of a non reduced 0-dimensional subscheme $Z\subset X_{m,d}$ of degree $4$. Clearly, for such a degree 4 scheme $Z$ we have that $\dim(\langle Z \rangle)\leq 3$. By \cite[Proposition 3.1]{ls}, or \cite[Subsection 3.2]{lt}, \cite[Remark 31]{bgi}, it is sufficient to do the cases $m=2,3$.  
The stratification of $\sigma_4(X_{1,d})$ is already known by \cite{cs} and \cite{bgi}. Hence it remains to study the stratification of $\sigma_4(X_{2,d})$ for $d\geq 5$ (in fact \cite{bgi} and \cite{kl} give it  for the cases $d=3,4$),
and the stratification of $\sigma_4(X_{m,d})$ for $m\geq 3$. What the already quoted results in \cite{bgi}, \cite{lt} and \cite{ls} allow to do is that, once we will have the stratification of $\sigma_4(X_{3,d})$ then we will straightforwardly have that the same stratification will hold for $\sigma_4(X_{m,d})$ for the same $d$ and $m\geq 3$.

By the discussion above we may  assume that the scheme $A$  defined in Notation \ref{Z} is not contained in a 2-dimensional projective subspace of $\mathbb{P}^m$. In fact, if $\deg(A)=4$ then $\langle A\rangle \subseteq \mathbb{P}^3$, but if $\langle A \rangle \subseteq \mathbb{P}^2$ then there exist a 0-dimensional scheme $B\subset \mathbb{P}^2$ of degree $3$ such that $\langle A \rangle \subseteq \langle B \rangle =\mathbb{P}^2$. This would imply that any point $P\in \langle\nu_{m,d}( A )\rangle$ belongs to $\langle\nu_{m,d}( B )\rangle$ for some 0-dimensional scheme $B\subset \mathbb{P}^m$ of degree 3. Now, since $\deg(B)=3$ then $\langle\nu_{m,d}( B )\rangle\subset\sigma_3(X)$. Therefore if $\langle A \rangle \subset \mathbb{P}^2$ we get that, if $Z=\nu_{m,d}(A)$, any point $P\in \langle Z \rangle$ belongs to $\sigma_3(X)$, but we want to study the $X$-rank of the points $P\in \sigma_4 (X)\setminus \sigma_3(X)$. Therefore we assume that the scheme $A\subset \mathbb{P}^m$ defined in Notation \ref{Z} spans a projective subspace of dimension $3$.

\begin{notation}\label{S} Let $P\in \sigma_4(X_{m,d})\setminus (\sigma_3(X_{m,d}) \cup \sigma_4^0(X_{m,d}))$.
We fix a finite set $S \subset X_{m,d}$ that
evinces the $X_{m,d}$-rank of $P$, i.e. $S\subset X_{m,d}$ is a finite set such that $P\in \langle S \rangle$ and $\sharp (S) =r_{X_{m,d}}(P)$ . Let $B\subset \mathbb {P}^m$ be the only set such that $\nu _{d,m}(B)=S$. We also say that
$B$ evinces the $X_{m,d}$-rank of $P$.
\end{notation}

This notation allows us to use many results on the reduced and non-reduced 0-dimensional schemes in $\mathbb{P}^m$ and translate them into informations on the 0-dimensional sub-schemes of $X_{m,d}$.

\section{Useful reducible curves}\label{S2}
Let $Y$ be a projective non-degenerate reduced curve obtained by the union of two rational normal curves $Y_1$, $Y_2$. We prove here two propositions on the $Y$-rank of points belonging to $\langle Z \rangle$ where $Z$ is a degree 4 non-reduced 0-dimensional sub-scheme of $Y$.
\\
We'd like to point out that in Proposition \ref{a1}  
we will only give an upper bound for the $Y$-rank of certain points, but in Section \ref{S4} it will be proved that
this upper bound is the actual value of the $Y$-rank
(Corollary \ref{ca1}).

\begin{proposition}\label{a1}
Fix an integer $d\ge 3$. Let $Y\subset \mathbb {P}^{2d}$ be a reduced and connected curve union of two smooth
degree $d$ curves $Y_1$, $Y_2$, each of them a rational normal curve in its linear span, with a unique
common point, $Q$, and with $\langle Y\rangle = \mathbb {P}^{2d}$. Let $Z\subset Y$ be a
0-dimensional scheme such that $\deg (Z)=4$, $Z_{red}= \{Q\}$, $Z$ is a Cartier divisor of $Y$ and $\deg(Z\cap Y_i)\geq 2$ for $i=1,2$. Fix
$P\in \langle Z\rangle$ such that $P\notin \langle Z'\rangle$ for any $Z'\subsetneq Z$. Then:
\begin{equation}\label{a1i}
r_Y(P)\leq 2d-2
\end{equation} 
and
there is a reduced 0-dimensional sub-scheme $S\subset Y$ such that $P\in \langle S\rangle$, $Q\notin S$, $\sharp (S) =2d-2$, and $\sharp (S\cap Y_i)=d-1$ for $i=1,2$. We may
find $S$ as above and not intersecting any finite prescribed subset of $Y$. 
\\
If $d\ge 4$,
then for a general pair of sets of points $(A_1,A_2) \subset Y_1\times Y_2$ such that $\sharp (A_1)=\sharp (A_2) =d-3$ there is $S$ as above with
the additional property that $A_1\cup A_2\subset S$.
\end{proposition}

\begin{proof}
\quad (a) First assume $d=3$. Let $\ell _P: \mathbb {P}^6\setminus \{P\}\to \mathbb {P}^5$ be the linear
projection from $P$. Since $P\notin Y$, then the map $\ell _P\vert Y$ is obviously a morphism. Since
$P\notin \langle Y_i\rangle$, each curve $C_i:= \ell _P(Y_i)$, $i=1,2$, is a rational normal curve in its $3$-dimensional
linear span. Since $\langle Y_1\rangle \cap \langle Y_2\rangle = \{Q\}$, the linear space $D:= \langle C_1\rangle \cap
\langle C_2\rangle$ has dimension at most $1$. Since $\dim (\langle C_i\rangle )=3$ for all $i$, we have $\dim (D)=1$.
Now the 0-dimensional sub-scheme $Z\subset Y$ is, by hypothesis, such that $P\in \langle Z\rangle$ and $P\notin \langle Z' \rangle$ for any proper sub-scheme $Z' \subset Z$. If $P\in \langle T_QY_1 \cup T_QY_2 \rangle$ we would have that $P$ belongs to the span of a proper sub-scheme of $Z$ of degree $3$ (in fact $\dim (\langle T_QY_1 \cup T_QY_2 \rangle )=2$, because $Q\in Y_1\cap Y_2$, while $\dim (\langle Z\rangle )=3$ because $d \ge 3$) that contradicts the hypothesis. 
Since $P\in \langle Z\rangle \setminus \langle T_QY_1\cup T_QY_2\rangle$, the line $D$ is not tangent either to $C_1$ or to $C_2$, but it intersects each $C_i$ only at their common point
$\ell _P(Q)$. Hence the linear projection from $D$ induces a degree $2$ morphism $\psi _i: C_i\to \mathbb {P}^1$. Thus,
for a general $O\in D$, there are two sets of points $B_i\subset C_i$ such that $\sharp (B_i)=2$
and $O\in \langle B_i\rangle$, for $i=1,2$. Let $S_i\subset Y_i$ be the only set of points such that $\ell _P(S_i) = B_i$ for $i=1,2$. Since
$\dim (\langle \ell _P(S_1\cup S_2)\rangle )=2$, we have $\dim (\langle \{P\}\cup S_1\cup S_2\rangle )=3$. We easily
find $O\in D$ such that $\dim (\langle S_1\cup S_2\rangle )=3$ and $Q\notin S_1\cup S_2$. Hence $P\in \langle S_1\cup
S_2\rangle$. We can then take $S:=S_1\cup S_2$ as a solution for $d=3$.

\quad (b) Now assume $d\ge 4$. Take a general pair of sets of points $(A_1,A_2) \subset Y_1\times Y_2$ such that $\sharp (A_1)=\sharp (A_2) =d-3$.
Let $\ell :\mathbb {P}^{2d}\setminus \langle A_1\cup A_2\rangle \to \mathbb {P}^6$ denote the linear projection from $\langle A_1\cup
A_2\rangle$.
Apply Step (a), i.e. the case $d=3$, to the curve $Y' \subset \mathbb {P}^6$ which is the closure
of $\ell (Y\setminus Y\cap \langle A_1\cup A_2\rangle )$. Let
$S_1\cup S_2$ be a solution for $Y'$ with respect to the point $\ell (P)$. For general $O\in D$ (as in Step (a))
we may find $S_1\cup S_2$ not through the finitely many points of $Y'$ which are in $Y'\setminus \ell (Y\setminus (A_1\cup
A_2))$. Hence there are unique
$B_i\subset Y_i$ such that $\ell (B_i)=S_i$ for $i=1,2$. Set $S:= B_1\cup B_2\cup A_1\cup A_2$. 
\end{proof}

In Corollary \ref{ca1} we will show that (\ref{a1i}) is an equality.

\begin{proposition}\label{a2}
Fix an integer $d\ge 3$. Let $Y\subset \mathbb {P}^{2d+1}$ be a reduced curve union of two smooth
degree $d$ curves $Y_1$, $Y_2$, each of them a rational normal curve in its linear span and such that $\langle Y_1\rangle \cap \langle Y_2\rangle =\emptyset$.
Fix $P_i\in Y_i$ for $i=1,2$. Let $Z_i\subset Y_i$ be the degree $2$ effective Cartier divisor $2P_i$ of $Y_i$, $i=1,2$. Set $Z:= Z_1\cup Z_2.$
Fix $P\in \langle Z\rangle$ such that $P\notin \langle Z'\rangle$ for all $Z'\subsetneq Z$. Then $b_Y(P)=4$, $Z$ is the only subscheme of $Y$ evincing $b_Y(P)$, 
$$r_Y(P) = 2d,$$
and $\sharp (S\cap Y_1) = \sharp (S\cap Y_2)=d$ for all reduced 0-dimensional sub-schemes $S\subset Y$ evincing $r_Y(P)$.
\end{proposition}

\begin{proof}
Since $\deg (Z)=4$ and since $Y$ is a smooth curve, we have $b_Y(P) \le 4$. 

Obviously $\langle Z\rangle \cap \langle Y_i\rangle \supseteq \langle Z_i\rangle$. Let's see the other containment. We show that $\langle Z \rangle \cap \langle Y_1 \rangle \subseteq \langle Z_1 \rangle$ (the same proof holds for $\langle Z \rangle \cap \langle Y_2 \rangle \subseteq \langle Z_2 \rangle$). If $\langle Z \rangle \cap \langle Y_1 \rangle $ is not contained in $ \langle Z_1 \rangle$, then there exists a point $Q\in \langle Z\rangle\cap \langle Y_1\rangle $ such that $Q \notin \langle Z_1\rangle$. Therefore $\langle Z_1 ,Q\rangle$ is a plane $\Pi$, and since
$Q\in \langle Z \rangle$, we have that 
$\langle Z_1 , Q \rangle:=\Pi\subset \langle Z \rangle$. Now, by construction, $Q\in \langle Y_1 \rangle$, hence $\Pi$ is spanned by a 0-dimensional scheme of degree 3 that is contained in $\langle Y_1 \rangle$, by hypothesis $\langle Y_1 \rangle \cap \langle Y_2 \rangle = \emptyset$, then $\Pi$ cannot intersect $\langle Z_2 \rangle$ which is entirely contained in $\langle Y_2 \rangle$. Both $\Pi$ and $\langle Z_2 \rangle$ are contained in $\langle Z \rangle$ which has projective dimension 3. Therefore, if such a $Q$ exists, we would have a projective plane $\Pi$ and a line $\langle Z_2 \rangle$ that are contained in a $\mathbb{P}^3$ without intersecting each other, but this is impossible. Then $\langle Z\rangle \cap \langle Y_i\rangle \subseteq \langle Z_i\rangle $ for $i=1,2$.

Since $P\notin \langle Z'\rangle$ for all $Z'\subsetneq Z$, we get $P\notin \langle Y_1\rangle$
and $P\notin 
\langle Y_2\rangle $. Since for $i=1,2$ $Y_i$ is a rational normal curve, then $Z_i$ is the only sub-scheme of $Y_i$ evincing $b_{Y_i}(Q)$ for
all $Q\in T_{P_i}Y_i\setminus \{P_i\}$. We immediately get that $Z$ is the only sub-scheme of $Y$ with length at most $4$ whose
linear span contains $P$. Hence we have proved that $Z$ is the unique 0-dimensional scheme that evinces the $Y$-border rank of $P$ and that $P\in \sigma_4(Y)\setminus \sigma_3(Y)$.

Now we compute $r_Y(P)$. 
\\
Let $\ell _P: \mathbb {P}^{2d+1} \setminus \{P\} \to \mathbb {P}^{2d}$ denote
the linear projection from $P$. Set $C:= \ell _P(Y)$ and $C_i:= \ell _P(Y_i)$. Since $P\notin \sigma _2(Y)$, then $\ell _P\vert Y$ is an embedding. Hence $C_1\cap C_2=\emptyset$. Since $P\notin \langle Y_1\rangle \cup  \langle Y_2\rangle $, then each
$C_i$ is a degree $d$ rational normal curve in its linear span. Thus $r_Y(P)$ is the minimal cardinality of a set of points $A:=A_1\cup A_2$ such that $A_1\subset C_1$, $A_2\subset C_2$ and $A_1\cup A_2$ is linearly dependent.
Notice that $\langle C_1\rangle \cap \langle C_2\rangle$ is a unique point $O\notin C$. Set $Q_i:= \ell _P(P_i)$ with $P_i=(Z_i)_{red}$, $W_i:= \ell _P(Z_i)$ and $W:= W_1\cup W_2$. Hence $W_i$ is the degree $2$ effective divisor $2Q_i$ of $C_i$. Since $P\in \langle Z\rangle$, then $\langle W\rangle$ is a plane. Thus the two lines $T_{Q_i}C_i$, $i=1,2$, meet each other. 
Since $\{O\} = \langle C_1\rangle \cap \langle C_2\rangle$, then $O$ is their unique common point. Since $O\in T_{Q_i}C_i$, we have $r_{C_i}(O)=d$ (see \cite{cs} or \cite[Theorem 4.1]{lt}). Hence $\sharp (A\cap C_1) \ge d$
and $\sharp (A\cap C_2)\ge d$. Since $\langle C_1\rangle \cap \langle C_2\rangle =\{O\}$ and any $d+1$ points of $C_i$ are linearly independent,
$\sharp (A\cap C_1) = d$
and $\sharp (A\cap C_2)=d$ for every linearly dependent $A\subset C$ such that $\sharp (A\cap C_i)\le d$ for all $i$. 
Then $r_Y(P)\geq 2d$, but $r_{Y_i}(R_i)=d$ for all $R_i\in T_{P_i}Y_i\setminus Y_i$ and $i=1,2$, hence $r_Y(P)\leq 2d$ and therefore $r_Y(P)=2d$.
\end{proof}

\section{Lemmas}\label{S3}

In  Notation \ref{S}  and in Notation \ref{Z} we defined  two different 0-dimensional schemes  $Z,S\subset X$ realizing the $X$-border rank and the $X$-rank respectively of a point $P\in \sigma_4(X)\setminus (\sigma_4^0(X)\cup \sigma_3(X))$ and two 0-dimensional schemes $A,B\subset \mathbb{P}^m$ such that $\nu_{m,d}(A)=Z$ and $\nu_{m,d}(B)=S$   respectively. Here, but only for this Section \ref{S3}, we do not care about the fact that $P\in \langle \nu_{m,d}(A) \rangle \cap \langle \nu_{m,d}(B) \rangle$ is a point of $\sigma_4(X)\setminus (\sigma_4^0(X)\cup \sigma_3(X))$: for this section $A,B\subset \mathbb{P}^m$ are 0-dimensional schemes whose images via $\nu_{m,d}$ still realize the $X$-border rank and the $X$-rank respectively of a point $P\in \mathbb{P}^n$, but here we do not give any restriction on the minimum secant variety $\sigma_s(X)$ such that $P\in \sigma_s(X)$. This is summarized in the following notation.

\begin{notation}\label{ABlemmi} In this section, and only in this section, when considering a point $P\in \mathbb {P}^n$, we assume that: 
\begin{itemize}
\item $A \subset \mathbb{P}^m$ is a non-reduced 0-dimensional scheme such that $\nu_{m,d}(A)=Z\subset \mathbb{P}^n$ evinces the $X$-border rank of $P\in \mathbb{P}^n$,
\item $B \subset \mathbb{P}^m$ is a reduced 0-dimensional scheme such that $\nu_{m,d}(B)=S\subset \mathbb{P}^n$ evinces the $X$-rank of $P\in \mathbb{P}^n$,
\item $\deg(A)<\deg(B)$.
\end{itemize}
More assumptions on the degrees of $A$ and $B$ will be explained in each Lemma.
\end{notation}

We recall the following result (\cite[Lemma 1]{bb}).

\begin{lemma}\label{v1} Fix $P\in \mathbb{P}^{n_{m,d}}$. Let $A,B\in \mathbb{P}^m$ be two 0-dimensional schemes as in Notation \ref{ABlemmi}.
Then $h^1(\mathbb {P}^m,\mathcal {I}_{A\cup B}(d)) >0$.
\end{lemma}

We introduce here a tool that we will use in the proofs of the next lemmata.
 
\begin{notation}
Let $E\subset \mathbb{P}^m$ be a 0-dimensional scheme and let $H\subset \mathbb{P}^m$ be a hyperplane, then the sequence that defines the residual scheme $\mbox{Res}_H(E)$ of $E$ with respect to $H$ is the following:
\begin{equation}\label{eqi1}
0 \to \mathcal {I}_{\mbox{Res}_H(E)}(t-1) \to \mathcal {I}_E(t) \to \mathcal {I}_{E\cap H,H}(t)\to 0.
\end{equation}
\end{notation}

\begin{lemma}\label{z1}
Fix an integer $d\ge 2$ and a 0-dimensional and curvilinear subscheme $E$ of $\mathbb {P}^2$ such
that $\deg (E)=2d+2$ and $h^1(\mathcal {I}_E(d))>0$. Assume that $E$ has at least one reduced component.

\quad (i) If $E$ is in linearly general position, then $h^1(\mathcal {I}_E(d))=1$ and
there is a smooth conic $C$ such that $E\subset C$.

\quad (ii) If $E$ is not in linearly general position, then either there is a line $L\subset \mathbb {P}^2$ such
that $\deg (L\cap E)\ge d+2$ and $h^1(\mathcal {I}_E(d)) = \deg (L\cap E) -d-1$ or there is a singular conic $\Gamma$ such that
$E\subset \Gamma$. The fact that $E$ has at least one reduced component, implies that $\Gamma$ is reduced, say $\Gamma = L_1\cup L_2$,
with $L_1\ne L_2$, $\sharp (E\cap L_1)=\sharp (E\cap L_2)=d+1$ and $E$ is a Cartier divisor of $L_1\cup L_2$.
\end{lemma}

\begin{proof}
First assume that $E$ is in linearly general position. Let $C\subset \mathbb {P}^2$ be a conic such that
$y:= \deg (E\cap C)$ is maximal. Since  $\deg (E)\ge 5$ and $\binom{4}{2}=6$,
we have $y\ge 5$. Since no line contains at least $3$ points of $E$, the conic $C$ is smooth. Since $\mbox{Res}_C(E)
\subset E$, the scheme $\mbox{Res}_C(E)$ is in linearly general position. Since
$\deg (\mbox{Res}_C(E)) =2d+2-y \le 2(d-2)+1$, we have $h^1(\mathcal
{I}_{\mbox{Res}_C(E)}(d-2))=0$ (\cite[Theorem 3.2]{eh}). Thus the exact sequence
$$0 \to \mathcal
{I}_{\mbox{Res}_C(E)}(d-2) \to \mathcal {I}_E(d) \to \mathcal {I}_{C\cap E}(d)\to 0$$
gives $h^1(C,\mathcal {I}_{C\cap E}(d))>0$. Thus $\deg (E\cap C)\ge 2d+2$. Since $\deg (E)=2d+2$, we get $E\subset C$, concluding the proof of (i).

Now assume that $E$ is not in linearly general position. Take a line $L\subset \mathbb {P}^2$ such that $x:= \deg (L\cap E)$
is maximal. By assumption we have $x\ge 3$. First assume
$x \ge d+2$. Since $\mbox{Res}_L(E)$ has degree $2d+2-x$, we have
$h^1(\mathcal {I}_{\mbox{Res}_L(E)}(d-1))=0$ (\cite[Lemma 34]{bgi}). From the exact sequence (\ref{eqi1}) we
get the result in this case. Now assume $x \le d+1$. If $h^1(\mathcal {I}_{\mbox{Res}_L(E)}(d-1))=0$, then (\ref{eqi1}) gives
$h^1(\mathcal {I}_E(d))=0$ that is a contradiction. Thus $h^1(\mathcal {I}_{\mbox{Res}_L(E)}(d-1))>0$. Since
$2d+2-x \le 2(d-1)+1$ (\cite[Lemma 34]{bgi}) gives the existence of a line $R$
such that $z:= \deg (R\cap \mbox{Res}_L(E))\ge d+1$. The maximality property of $x$ and
the inclusion $\mbox{Res}_L(E) \subseteq E$ gives $x\ge d+1$. Since $z\le 2d+2-x$, we get $z=x=d+1$. If $E$ has at least one reduced component, then $L\ne R$, because $\deg (\mathrm{Res}_L(E) ) = \deg (E\cap L)$ and $E$ has at
least one reduced connected component.  A Weil divisor $W$ 
of $L_1\cup L_2$ is locally principal, except at most at $L_1\cap L_2$. A well-known property of nodal singularities says
that $W$ is Cartier if and only if $\deg (W) = \deg (W\cap L_1)+\deg (W\cap L_2)$. This condition is satisfied by $E$.
\end{proof}

We remind here the part of \cite[Theorem 1]{bb}, that will be useful in our paper applied in the particular case of $\deg (Z)=4$.

\begin{lemma}\label{bbb}  
Assume $m \ge 2$ and let $E,F\subset \mathbb{P}^m$ be two 0-dimensional schemes such that there exists a point $Q\in \langle
\nu_{m,d}(E)\rangle \cap \langle \nu_{m,d}(F)\rangle$ such that $Q\notin \langle \nu_{m,d}(E')\rangle$ for any $E'\subsetneq
E$ and $Q\notin \langle \nu_{m,d}(F')\rangle$ for any $F'\subsetneq F$. Then there are a line
$L\subset
\mathbb {P}^m$ and a finite set of points
$F_2\subset
\mathbb {P}^m\setminus L$ such that $\deg (L\cap (E\cup F)) \ge d+2$, $E = F_2\sqcup (E\cap L)$ and $F = F_2\sqcup (F\cap L)$
(as schemes).
\end{lemma}

\section{Preliminaries of the proof of the main theorem}\label{S4}

This section is essentially the core of the proof of Theorem \ref{theorem} but it is not the proof yet. Nevertheless that proof will be done in the next section. Here we give only all  the preliminaries in such a way that  the proof of Theorem \ref{theorem} will be reduced to its structure.

\begin{lemma}\label{c2} 
Let $E, F\subset \mathbb {P}^m$ be 0-dimensional schemes such that
$h^1(\mathcal {I}_E(d)) = h^1(\mathcal {I}_F(d))=0$.
Let
$M
\subset
\mathbb {P}^m$ be a hyperplane such that $h^1(\mathbb {P}^m,\mathcal {I}_{\mbox{Res}_M(E\cup F)}(d-1)) =0$. Then
$h^1(\mathbb{P}^m, \mathcal {I}_{E\cup F}(d))= h^1(M,\mathcal {I}_{(E\cup F)\cap M}(d))$.

\quad (a) If $\mbox{Res}_M(E)\cap \mbox{Res}_M(F) = \emptyset$, then $\langle \nu _{m,d}(E)\rangle
\cap \langle \nu _{m,d}(F)\rangle = \langle \nu _{m,d}(E\cap M)\rangle
\cap \langle \nu _{m,d}(E\cap M)\rangle 
\subseteq
\langle \nu_{m,d}(M)\rangle$. 

\quad (b) If $\mbox{Res}_M(E)\cap \mbox{Res}_M(F) \ne \emptyset$ and $F$ is reduced, then $\langle \nu _{m,d}(E)\rangle
\cap \langle \nu _{m,d}(F)\rangle$  is the linear span of $\langle \nu _{m,d}(E\cap M)\rangle \cap \langle \nu _{m,d}(F\cap M)\rangle$ and of the set $\nu _{m,d}(\mbox{Res}_M(E)_{red}\cap (F\setminus F\cup M))$.
\end{lemma}

\begin{proof}
Since $h^2(Y_m,\mathcal {I}_{E\cup F}(d-1)) =0$, the first equality follows from the residual sequence (\ref{eqi1})
for the scheme $E\cup F$. 

Assume
$\mbox{Res}_M(E)\cap \mbox{Res}_M(F) = \emptyset$, i.e. $E\cap F \subset M$. Since $h^1(\mathbb{P}^m, \mathcal {I}_{E\cup
F}(d))= h^1(M,\mathcal {I}_{(E\cup F)\cap M}(d))$, we have $\dim (\langle \nu _{m,d}(E\cup F)\rangle )
= \deg (E\cup F)-1-h^1(\mathbb{P}^m, \mathcal {I}_{E\cup F}(d))$, i.e.
$\dim (\langle \nu_{m,d}(E)\rangle \cap \langle \nu _{m,d}(F)\rangle ) = \deg (E\cap F)-1 +h^1(\mathbb{P}^m, \mathcal
{I}_{E\cup F}(d))$. For the same reason we have
$\dim (\langle \nu_{m,d}(E\cap M)\rangle \cap \langle \nu _{m,d}(F\cap M)\rangle ) = \deg (E\cap F\cap M)-1
+h^1(\mathbb{P}^m,
\mathcal {I}_{E\cup F}(d))$. Hence $\dim (\langle \nu _{m,d}(E)\rangle
\cap \langle \nu _{m,d}(F)\rangle ) = \dim (\langle \nu _{m,d}(E\cap M)\rangle
\cap \langle \nu _{m,d}(F\cap M)\rangle )$. Hence $\langle \nu _{m,d}(E)\rangle
\cap \langle \nu _{m,d}(F)\rangle  = \langle \nu _{m,d}(E\cap M)\rangle
\cap \langle \nu _{m,d}(F\cap M)\rangle$.

If $\mbox{Res}_M(E)\cap \mbox{Res}_M(F) \ne \emptyset$,
then we need to add its degree to compute $\dim (\langle \nu _{m,d}(E)\rangle
\cap \langle \nu _{m,d}(F)\rangle )$. If $F$ is reduced, then $\deg (\mbox{Res}_M(E)\cap \mbox{Res}_M(F))
= \sharp (\mbox{Res}_M(E)_{red}\cap (F\setminus F\cup M))$.
\end{proof}

Now we split the section in two subsections where we study the $X_{m,d}$-rank of a point $P\in \langle \nu_{m,d}(A) \rangle$ for particular configurations of the scheme $A\subset \mathbb{P}^m$ with $m=2,3$ respectively (if $A\subset \mathbb{P}^1$ we refer to the Sylvester algorithm in \cite{cs}, \cite{bgi}, \cite{cglm} and \cite{bcmt} for the computation of the $X_{1,d}$-rank of a point $P\in \langle\nu_{1,d}(\mathbb{P}^1) \rangle$). 

\subsection{Two dimensional case}\label{m2} 

Here we study the $X_{2,d}$-rank of a point $P\in \sigma_4(X_{2,d})\setminus (\sigma_4^0(X_{2,d})\cup \sigma_3(X_{2,d}))$ with $X_{2,d}$ the Veronese surface $\nu_{2,d}(\mathbb{P}^2)\subset \mathbb{P}^{n_{2,d}}$. Moreover we assume in this sub-section that the scheme $A\subset \mathbb{P}^2$ such that $Z=\nu_{2,d}(A)$ evinces the $X_{2,d}$-border rank of $P$ is not contained in a line, that is to say that $m=2$ is the minimum dimension of a linear space that contains $A$ where $A$ is defined as in Notation \ref{Z}. 
Since $A$ is not contained in a line we have that $\langle A\rangle=\mathbb{P}^2$ and $h^0(\mathbb{P}^2, \mathcal{I}_A(2))=2$.

\subsubsection{Here we assume the existence of a line $L \subset \mathbb{P}^2$ such that the schematic intersection between $A$ and $L$ has degree at least 3}\label{a} \ \\
Since we are assuming that there exists a line $L\subset \mathbb{P}^2$ such that $\deg(A\cap L) \ge 3$ 
and since $A\nsubseteq L$ we have $\deg (A\cap L) =3$ and $\mbox{Res}_L(A)$ is a point, $O$, with its reduced structure.
\\
Notice that every point $P'$ of $\langle \nu_{2,d}(A\cap L)\rangle\setminus\sigma_2(\nu_{2,d}(L))$ has rank $d-1$ (\cite{cs} or \cite[Theorem 4.1]{lt}), unless $A\cap L$ is reduced. In the latter case
any such a point has rank $3$.  

In Proposition \ref{la1} we study the case of $O\notin L$, while the case of $O\in L$ is done in Proposition \ref{la1.1}.

\begin{proposition}\label{la1} Let $A\subset \mathbb{P}^2$ be a 0-dimensional non-reduced scheme of degree $4$. Assume the existence of a line $L\subset \mathbb{P}^2$ such that $deg(L\cap A)\geq 3$ and $\mathrm{Res}_L(A)=:\{O\}\notin L$. Then
$$r_{X_{2,d}}(P) = d$$for every $P\in \langle \nu _{2,d}(A) \rangle\setminus (\sigma_4^0(X_{2,d})\cup\sigma_3(X_{2,d}))$.
\end{proposition}

\begin{proof}
Since $\mbox{Res}_L(A) \ne \emptyset$, we have $A\nsubseteq L$, i.e. $\deg (A\cap L) = 3$. Hence $A = (A\cap L)\sqcup \{O\}$
with $O\notin L$. Since $P\notin \langle \nu _d(A\cap L) \rangle$, the line $\langle \{\nu _{2,d}(O),P\}\rangle \subset
\langle \nu _{2,d}(A)\rangle$ meets the plane $\langle \nu _{2,d}(A\cap L) \rangle$
in a unique point $P'$. We have $P'\in \langle \nu_{2,d} (A\cap L) \rangle$, a theorem
of Sylvester gives $P'\in \sigma_3(\nu_{2,d}(L))\setminus (\sigma_3^0(\nu_{2,d}(L))\cup \sigma_2(\nu_{2,d} (L)))$ and $r_{\nu
_{2,d}(L)}(P') =d-1$ (\cite{cs}, \cite{cglm}, \cite{bcmt}, \cite[Theorem 4.1]{lt}, \cite[\S 3]{bgi}). Hence $r_X(P')\le d-1$.
Since $P\in \langle \{\nu _{2,d}(O), P'\}\rangle$, we have $r_X(P) \le r_X(P')+1 = d$.

Assume $r_{X}(P) < d$, i.e. $\sharp (B) \le d-1$. Hence
$\deg (A\cup B) \le d+3 \le 2d+1$. There is a line $R\subset \mathbb{P}^2$ such that $\deg ((A\cup B)\cap R) \ge d+2$ (\cite[Lemma 34]{bgi}).
\\ First assume $R = L$. Since $P\notin \langle \nu _{2,d}(L)\rangle$ and $P\in \langle \nu _{m,d}(B)\rangle$, we have $B\cap L\subsetneq B$. Hence $\sharp (B\cap L) \le d-2$. Hence $\deg (A\cap L)\ge 4$, a contradiction. 
\\
Now assume $R\ne L$. Since the scheme $R\cap L$ is a reduced
point, we have $\deg (L\cap A\cap R) \le 1$. Since
$\deg (A\cap L) =\deg (A)-1$, we get
$\deg (A\cap R) \le 2$. Hence $\deg (R\cap B)\ge d$, contradicting
the assumption $\sharp (B)<d$.
\end{proof}

In the next two propositions we will do the cases in which the point $O=\mathrm{Res}_L(A)$ is contained in $L$. 
Observe that the definition of the residual scheme shows that the connected component $A_O$ of $A$ containing $O$ is not reduced. We will distinguish the three propositions below by the cardinality of the support of $A$.
 
\begin{proposition}\label{la1.1}  Let $A\subset \mathbb{P}^2$ be a 0-dimensional and connected scheme of degree $4$ such that there is a line $L\subset \mathbb{P}^2$ with $deg(L\cap A)\geq 3$ and $\mathrm{Res}_L(A)=:\{O\}\in L$.
If $A$ is curvilinear, then$$r_{X_{2,d}}(P) = 2d-2$$for every $P\in \langle \nu _{2,d}(A) \rangle\setminus (\sigma_4^0(X_{2,d})\cup\sigma_3(X_{2,d}))$.
If $A$ is not curvilinear, then $\langle \nu _{2,d}(A) \rangle \subset \sigma_3(X_{2,d})$.\end{proposition}

\begin{proof}
Since $\mathrm{Res}_L(A)\ne \emptyset$, we have $A\nsubseteq L$. Hence $\deg (A\cap L)=3$. If $A$ is not curvilinear, then we are in case (ii) of Lemma \ref{oo1} with as $x$ a local equation of $L$.
Hence $\langle \nu _{2,d}(A) \rangle \subset \sigma_3(X_{2,d})$ (Lemma \ref{oo1} and Remark \ref{ee4}). Hence from now on, we assume that $A$ is curvilinear. We first prove that $r_{X_{2,d}}(P) \ge 2d-2$. Assume that 
$r_{X_{2,d}}(P) \le 2d-3$. Hence $\deg (A\cup B) \le 2d+2$. Therefore we may apply Lemma \ref{bbb}. Since $A_{red}$ is a single point, Lemma \ref{bbb} gives $F_2=\emptyset$, i.e. $A\cup B
\subset L$, a contradiction.

Now we prove that  $r_{X_{2,d}}(P) \le 2d-2$. Let $L'\subset \mathbb {P}^2$ be any line such that $O\in L'$ and $L'\ne L$. Since $\mathrm{Res}_L(A) \in L'$
we have $A\subset L\cup L'$. We will find $E\subset L\cup L'$ such that $\sharp (E) =2d-2$, $P\in \langle \nu _{2,d}(E)\rangle$, $\sharp (E\cap L)=d-1$,
$\sharp (E\cap L')=d-1$ and $O\notin E$. Notice that $\langle \nu _{2,d}(L)\rangle$ and  $\langle \nu _{2,d}(L')\rangle$ are $d$-dimensional linear subspaces of $\mathbb {P}^n$
with $\nu _{2,m}(O)$ as its intersection. Fix any $Q\in \langle \nu _{2,d}(A\cap L)\rangle$ such that $Q\notin \langle \nu _{2,d}(v)\rangle$, where $v$ is the tangent vector of $L$ at $O$.
We have $r_{\nu _{2,d}(L)}(Q) =d-1$ (\cite{cs}, \cite[Theorem 4.1]{lt}, or \cite[\S 3]{bgi}). Fix $E_1\subset L$ such that $\nu _{2,d}(E_1)$ evinces  $r_{\nu _{2,d}(L)}(Q)$.
Since $P\notin \langle \nu _{2,d}(L)\rangle $ and $P\notin \langle \nu _{2,d}(L')\rangle )$ (e.g. because the scheme $A$ is not contained in a line), the line $\langle \{Q, P\}\rangle$ meets $\langle \nu _{2,d}(L')\rangle$ at a unique
point, $Q'$. Take $E_2\subset L'$ such that $\nu _{2,d}(E_2)$ evinces $r_{\nu _{2,d}(L')}(Q')$ and set $E:= E_1\cup E_2$. It is sufficient to prove that $\sharp (E_2)\le d-1$.
Assume $\sharp (E_2) \ge d$. Hence $\nu _{2,d}(Q')$ is contained in the tangent developable of $\nu _{2,d}(L')$ (\cite{cs}, \cite[Theorem 4.1]{lt}, or \cite[\S 3]{bgi}).
Hence there is a degree $2$ scheme $W\subset L'$ such that $Q' \in \langle \nu _{2,d}(W)\rangle$. Hence $P\in \langle \nu _{2,d}(W \cup E_1)\rangle$. Since $\deg (\mbox{Res}_L(W\cup E_1\cup A))
\le 2$, we have $h^1(\mathcal {I}_{\mbox{Res}_L(A\cup W\cup E_1)}(d-1)) =0$. Hence the proof of Lemma \ref{c2} gives $\dim (\langle \nu _{2,d}(A)\rangle \cap \langle \nu _{2,d}(W\cup E_1)\rangle)
= \dim (\langle \nu _{2,d}(A\cap L)\cap \langle \nu _{2,d}((W\cup E_1)\cap L)\rangle )+ \deg (A\cap (W\cup E_1)) -\deg ((A\cap L)\cap  ((W\cup E_1)\cap L))$. We have
$E_1\cap A =\emptyset$ and $\deg (W\cap L) \le 1$, with equality only if $W_{red} = O$. Since $W\cap A=\emptyset$ if $W_{red} \ne O$, it is sufficient
to notice that $\deg (W\cap A)=1$ if $W_{red}=O$, because $L'\ne L$ and $L$ is the Zariski tangent space of the curvilinear scheme $A$.
\end{proof}

\begin{lemma}\label{m1}
Let $D, R\subset \mathbb {P}^2$  be two different lines intersecting in $O\in \mathbb{P}^2$. 
Fix $P_1\in R\setminus \{O\}$ and let $A_2\subset R$ be the degree $2$
effective divisor of $R$ with $P_1$ as its support. 
Let $A_1\subset D$ be the degree $2$ effective divisor of $D$ with $O$ as its support.
Set $A:= A_1\cup A_2$.
Fix $P\in \langle \nu _{2,d}(A)\rangle$ such that $P\notin \langle \nu _{2,d}(A')\rangle$ for any
$A'\subsetneq A$. Then:

\quad (a) $r_{X_{2,d}} (P) =2d-2$.

\quad (b) Fix any $B\subset \mathbb {P}^2$ evincing $r_{X_{2,d}} (P)$. Then $O\notin B$, $P_1\notin B$, $\sharp
(B\cap R)=d-2$ and there is a line $L\subset \mathbb {P}^2$ such that $O\in L$, $L\ne D, R$,
$\sharp (B\cap L )=d$ and $B\subset L\cup R$.

\quad (c) Fix any line $L\subset \mathbb {P}^2$ such that $O\in L$, $L\ne D,R$.
Then there is $B\subset \mathbb {P}^2$ evincing $r_{X_{2,d}} (P)$ and such that $\sharp (B\cap L)
= d$.
\end{lemma}

\begin{proof}
Fix any line $L\subset \mathbb {P}^2$ passing through $O$ and such that $L\ne D,R$. We first prove the
existence of a finite set $E\subset \mathbb {P}^2$ such that $\sharp (E) =2d-2$, $O\notin E$, $P_1\notin E$, $\sharp
(E\cap R)=d-2$, $\sharp (E\cap L) =d$ and $P\in \langle \nu _{2,d}(E)\rangle$. Let $2O$ be the first
infinitesimal neighborhood of $O$ in $\mathbb {P}^2$. Let $E_1, E_2, A_1 \in \mathbb{P}^2$ be 0-dimensional schemes obtained by intersecting  $2O$ with $L,R$ and $D$  respectively.
Since $d\ge 4$, we have $\dim (\langle \nu _{2,d}(2O\cup A_2)\rangle )=4$.
Fix any $Q_1\in \langle \nu _{2,d}(E_1)\rangle \setminus \{\nu _{2,d}(O)\}$. Since $\langle \nu _{2,d}(E_1)\rangle$
is the tangent line at $\nu _{2,d}(O)$ of the rational normal curve $\nu _{2,d}(L)$, a theorem of Sylvester
gives the existence of $B_1\subset L\setminus \{O\}$ such that $Q_1\in \langle \nu _{2,d}(B_1)\rangle$ (\cite{cs}, \cite{bgi}, \cite{bcmt}). 
The $4$-dimensional linear space $\langle \nu _{2,d}(2O\cup A_2)\rangle$ contains 
$\langle \nu _{2,d}(E_2\cup A_2)\rangle$. Hence the line $\langle \{P,Q_1\}\rangle \subset\ \langle \nu _{2,d}(2O\cup A_2)\rangle$ contains at least one point, $Q_2$, of $\langle \nu _{2,d}(E_2\cup A_2)$.

\quad {\emph {Claim:}} For general $Q_1\in \langle \nu_{2,d}(E_1)\rangle \setminus \{ \nu_{2,d}(O)\}$ we may find $Q_2\in \langle \{P, Q_1\} \rangle \cap \langleÊ\nu_{2,d} (E_2 \cup A_2)\rangle $ such that $Q_2\notin \langle
\nu _{2,d}(F)\rangle$ for any $F\subsetneq E_2\cup A_2$.

\quad {\emph {Proof of the Claim:}} There are only $2$ degree $3$ subschemes of $E_2\cup A_2$
($\{O\}\cup A_2$ and $E_2\cup \{P_1\}$) and it is sufficient to prove the claim for these subschemes. Assume the claim fails for one of them. Hence $\{P,Q_1\}\subset \langle \nu _{2,d}(F)\rangle$
with either $F = \{O\}\cup A_2$ or $F = E_2\cup \{P_1\}$.  In particular $P\in \langle \nu _{2,d}(F)\rangle$. The case  $F = \{O\}\cup A_2$ contradicts the hypothesis  ``~$P\notin \langle \nu _{2,d}(A')\rangle$ for any
$A'\subsetneq A$~", in fact $F$ is a proper subscheme of  $A$.
Now assume $F=E_2\cup \{P_1\}$.
There is $Q_4\in \langle \nu _{2,d}(E_2)\rangle$ such that $P\in \langle Q_4\cup \nu _{2,d}(P_1)\rangle$. We get $r_{X_{2,d}}(P) \le d+1$.
The proof of parts (a) and (b) below (the line $D$) shows that  $r_{X_{2,d}}(P) \le d+1$
is not even possible (alternatively the contradiction  can be inferred from Lemma  \ref{bbb} because $d+1 + 4 \le 2d+1$).

By the Claim and the quoted theorem of Sylvester there is $B_2\subset R$ such
that $\sharp (B_2)=d-2$ and $Q_2\in \langle \nu _{2,d}(B_2)\rangle$.

Set $E:= B_1\cup B_2$. Since $Q_i\in \langle \nu _{2,d}(B_i)\rangle$, $i=1,2$, $Q_2
\in \langle \{P,Q_1\}\rangle$ and $Q_1\ne Q_2$, we have $P\in \langle \nu _{2,d}(E)\rangle$.

In particular we proved that $r_{X_{2,d}} (P)\le 2d-2$. Let $B\subset \mathbb{P}^2$ be a set evincing $r_{X_{2,d}}(P)$.  Since $r_{X_{2,d}}(P) \le 2d-2$, we
have $\deg (A\cup B) \le 2d+2$.

 First assume $\deg (A\cup B) \le 2d+1$. Since no component of $A$ is reduced, Lemma \ref{bbb}
gives that $A$ is contained in a line, a contradiction. 

Now assume $\deg (A\cup B)=2d+2$. Hence $A\cap B = \emptyset$. Since
$h^1(\mathbb{P}^2,\mathcal {I}_{A\cup B}(d))>0$ (Lemma \ref{v1}), either
there is a line $M\subset \mathbb {P}^2$ such that $\deg (M\cap (A\cup B)) \ge d+2$
or there is a conic $T$ such that $\deg (T\cap (A\cup B)) \ge 2d+2$ (Lemma \ref{z1}).

First assume the existence of a line $M$ such that $\deg (M\cap (A\cup B)) \ge d+2$.
Since $\deg (\mbox{Res}_M(A\cup B))\le d$, we have
$h^1(\mathcal {I}_{\mbox{Res}_M(A\cup B)}(d-1))=0$.
Since no connected component of $A$ is reduced, Lemma \ref{bbb}  gives $A\subset L$, a contradiction.

Now assume the existence of a conic $T$ such that $\deg (T\cap (A\cup B)) \ge 2d+2$.
Since $\deg (A\cup B) \le 2d+2$, we get $\deg (A\cup B)=2d+2$ and $A\cup B\subset T$.
Since $\deg (A\cup B) =2d+2$, we have $O\notin B$ and $P_2\notin B$.
Since $\deg (A\cap R)=3$, Bezout theorem gives $T = R\cup L$ with $L$
a line. Since $A_1\subset T$ and $A_1\nsubseteq R$, we have $O\in L$. We
just checked that $\deg ((A\cup B)\cap R) \le d+1$ and $\deg ((A\cup B)\cap L)
\le d-1$. Since $\deg (A_1\cap D) = 2$, we get $L\ne D$, $\sharp (B\cap L)=d$
and $\sharp (B\cap R)=d-2$.
\end{proof}

\begin{corollary}\label{m2}
Let $A\subset \mathbb{P}^2$ be a 0-dimensional scheme of degree $4$ obtained as the union of two non-reduced degree 2 schemes $A_1,A_2$ with supports on the same  line $ R= \langle A_2 \rangle$, and such that $D:= \langle A_1\rangle \neq R$. Set $O:= (A_1)_{red}$ and $O':= (A_2)_{red}$. Let $P \in \langle  \nu_{2,d} (A)\rangle$ be such that $P\notin \langle \nu _{2,d}(A')\rangle$
for any $A'\subsetneq A$.
Set $Y:= \nu _{2,d}(R\cup L)$. Then 
$$r_Y(P) = 2d-1$$
and there is $E\subset L\cup R$ such that $\nu _{2,d}(E)$ evinces $r_Y(P)$,
$O\notin E$, $O'\notin E$, $\sharp (E\cap D) = d$ and $\sharp (E\cap R)=d-1$.
\end{corollary}

\begin{proof}
Lemma \ref{m1} gives $r_{X_{2,d}} (P) = 2d-2$ and that no set $B\subset \mathbb {P}^2$ evincing
$r_{X_{2,d}} (P)$ is contained in $Y$. Hence it is sufficient to prove the existence of
$E\subset L\cup R$ such that $\sharp (E) =2d-1$, $O\notin E$, $O'\notin E$,
$\sharp (E\cap D)=d$, $\sharp (E\cap R)=d-1$
and $P\in \langle \nu _{2,d}(E)\rangle$. Fix a general $Q_1\in \langle \nu _{2,d}(A_1)\rangle$
and take $E_1\subset D\setminus \{O\}$ such that $Q_1\in \langle \nu _{2,d}(E_1)\rangle$ (Sylvester).
Notice that $\langle \nu_{2,d}(\{O\}\cup A_2)\rangle$ is a hyperplane
of the $3$-dimensional linear space $\langle \nu _{2,d}(A)\rangle$. Hence
the line $\langle \{Q_1,\nu _{2,d}(O')\}\rangle \subset \langle \nu _{2,d}(A)\rangle$
contains a point, $Q_2$, of $ \langle \nu_{2,d}(\{O\}\cup A_2)\rangle$. Since
$r_{X_{2,d}}(P) > d+2$, we have $Q_2\notin \langle \nu _d(\{O,O'\})\rangle$.
It is easy to check that we may find $Q_1$ such that $Q_2\notin \langle \nu _{2,d}(A_2)\rangle$.
Hence $r_{\nu _{2,d}(R)}(Q_2) =d-1$ (Sylvester). Take $B_2\subset R$ such that
$\nu _{2,d}(B_2)$ evinces $r_{\nu _{2,d}(R)}(Q_2)$. Sylvester's theorem also gives $B_2\cap \{O,O'\} =\emptyset$. Set $E:= B_1\cup B_2$.
\end{proof}

\begin{lemma}\label{m3lemma}
Fix a line $R\subset \mathbb {P}^2$, $O\in R$, and a connected 0-dimensional
and curvilinear scheme $A \subset \mathbb{P}^2$ such that $\deg (A)=4$, $A_{red} = \{O\}$ and $\deg (R\cap A)=3$.
Fix $P\in \langle \nu _{2,d}(A)\rangle$ such that $P\notin \langle \nu _{2,d}(A')\rangle$ for any
$A'\subsetneq A$. Then:

\quad (a) $r_{X_{2,d}} (P) =2d-2$.

\quad (b) Fix any $B\subset \mathbb {P}^2$ evincing $r_{X_{2,d}} (P)$. Then $O\notin B$, $\sharp
(B\cap R)=d-2$ and there is a line $L\subset \mathbb {P}^2$ such that $O\in L$, $L\ne R$,
$\sharp (B\cap L )=d$ and $B\subset L\cup R$.

\quad (c) Fix any line $L\subset \mathbb {P}^2$ such that $O\in L$ and $L\ne R$.
Then there is $B\subset \mathbb {P}^2$ evincing $r_{X_{2,d}} (P)$ and such that $\sharp (B\cap L)
= d$.
\end{lemma}

\begin{proof}
Let $A_4\subset R$ be the degree $4$ effective divisor of $R$ with $O$ as
its support.
We modify the proof of Lemma \ref{m1} in the following way.
Notice that the scheme $2O\cup A_4$ has degree $5$. We use the
$4$-dimensional
linear space $\langle \nu _{2,d}(2O\cup A_4)\rangle$ instead of
$\langle \nu _{2,d}(2O\cup A_2)\rangle$ and the hyperplane $\langle \nu
_{2,d}(A_4)\rangle$
of $\langle \nu _{2,d}(2O\cup A)\rangle$ instead of $\langle \nu
_{2,d}(E_2\cup A_2)\rangle$.
\end{proof}

\begin{proposition}\label{la1.3} Assume $d\ge 4$. Take $ A = A_O\sqcup O_1\sqcup O_2\subset \mathbb{P}^2$ with $O_1\ne O_2$ two simple points of $\mathbb{P}^2$ and $A_O\subset \mathbb{P}^2$a degree $2$ non-reduced 0-dimensional scheme with support on a point $O\in L:=\langle O_1, O_2\rangle$ but $O\notin \{O_1, O_2\}$ and $\deg(A_O\cap L)=1$. Set $Z=\nu_{2,d}(A)$. Then
$$r_{X_{2,d}}(P) = d+2$$
for every $P\in \langle Z \rangle\setminus (\sigma_4^0(X_{2,d})\cup\sigma_3(X_{2,d}))$.\end{proposition}

\begin{proof} Define $Z_O:=\nu_{2,d}(A_O)\subset X_{2,d}$. Every point $P'\in \langle \nu_{2,d}(Z_O)\rangle \setminus X_{2,d}$ has $X_{2,d}$-rank equal to $d$ (see \cite[Theorem 32]{bgi}).
Thus $r_X(P) \le d+2$ in this case. 
Assume $r_X(P) \le d+1$. Since $d+5 \le 2d+1$ (here we are using the hypothesis  $d \ge 4$), we may  apply Lemma \ref{bbb}. We get the existence
of a line $R\subset Y_2$ and of a set of points $F_2\subset Y_2\setminus R$ such that $\deg ((A\cap B)\cap R) \ge d+2$, $\sharp (F_2) \ge 1$, $B = ((B\setminus B\cap A_{red})\cap R)\sqcup F_2$, $A\cap R\ne \emptyset$, $B\cap R\ne \emptyset$, $B = (B\cap R)\sqcup F_2$ and $A = (A\cap R)\sqcup F_2$ where $B$ is as in Notation \ref{S}.
First assume $R=L$. Since $A_{red} \subset L$, we get $F_2 =\emptyset$ and hence $A\subset L$, a contradiction. Now
assume $R \ne L$. Thus $\{O\} = R\cap L$, $A_O$ is the degree $2$ effective divisor of $R$ supported by $O$ and $F_2 =
\{O_1,O_2\}$. Since $P\notin \langle \nu_{2,d}(O_1),\nu_{2,d}(O_2), \nu_{2,d}(O)\rangle$ (in fact we have assumed that $Z=\nu_{2,d}(A)$ evinces the $X$-border rank of $P$ and $\deg(Z)=4$), we have $\langle \nu_{2,d}(A\cap L)\rangle
\cap \langle \nu_{2,d}(A_O)\rangle
\subseteq \langle (B\cap R)\setminus \{O\}\rangle$. Since $r_{\nu_{2,d}(R)}(U) = d$ for all $U\in \langle \nu_{2,d}(A_O)\rangle \setminus \{O\}$ (see \cite{cs}), we get $\sharp ((B\cap R)\setminus \{O\})\ge d$. Thus $\sharp (B)\ge d+2$, a contradiction.
\end{proof}

\begin{remark}\label{2a1.3}
Take $m \ge 2$ and $A = A_O\sqcup O_1\sqcup O_2\subset \mathbb{P}^m$ with $A_O\subset \mathbb{P}^m$ connected and $\deg (A_O)=2$ and $O_1,O_2\in \mathbb{P}^m$. Notice that if $m>2$ we are not assuming that $A$ is contained in a plane. As in Proposition \ref{la1.3} if $P\in \langle \nu_{m,d}(A_O\sqcup O_1\sqcup O_2)\rangle\setminus (\sigma_4^0(X_{m,d})\cup\sigma_3(X_{m,d}))$, then
$$r_{X_{m,d}}(P) = d+2.$$
Let $L \subset \mathbb {P}^m$ be the line spanned by $A_O$. Set $\{O\}:= (A_O)_{red}$. Let $T$ be the tangent line to the degree $d$ rational normal curve $\nu_{m,d}(L)$ at $\nu_{m,d}(O)$.
The plane $\langle \{\nu_{m,d}(O_1),\nu_{m,d}(O_2),P\}\rangle$ intersects $T$ at a unique point $P_1$ and $P_1\ne \nu_{m,d}(O)$. Hence $r_{\nu_{m,d}(L)}(P_1)=d$.
Using Sylvester's algorithm (see \cite[\S 3]{bgi}) to find a set $S_1\subset \nu_{m,d}(L)$ evincing $r_{\nu_{m,d}(L)}(P_1)$. The set $S_1\cup \{\nu_{m,d}(O_1),\nu_{m,d}(O_2)\}$ evinces
$r_{ X_{m,d}}(P)$.
\end{remark}

This concludes our Subsection \ref{a} in which we were assuming the existence of a line $L\subset \mathbb{P}^2$ such that $\deg(A\cap L)\geq 3$.

\subsubsection{Here we assume $\deg (A\cap L) \le 2$ for every line $L\subset \mathbb{P}^2$}\label{b}

The assumption 
 ``~$\deg (A\cap L) \le 2$ for every line $L\subset \mathbb{P}^2$~'' is equivalent to the spannedness of the sheaf $\mathcal {I}_A(2)$.

Let's do first the case in which the generic conic $E\in |\mathcal{I}_A(2)|$ is smooth.

\begin{proposition}\label{lb1} Assume $d\ge 4$. Let $A\subset \mathbb{P}^2$ be a zero dimensional scheme of degree $4$ such that $\mathcal {I}_A(2)$ is spanned and $\sharp (A_{red}) \le 2$. Suppose that the general conic $E\in \vert \mathcal {I}_A(2)\vert$ is smooth.
Then$$r_{X_{2,d}(P)}=2d-2$$for every  $P\in \langle \nu _{2,d}(A) \rangle \setminus (\sigma_4^0(X_{2,d})\cup\sigma_3(X_{2,d}))$.
\end{proposition}

\begin{proof}
Notice that
$Y:= \nu_{2,d}(E)$ is a degree
$2d$ rational normal curve in its linear span. Let $B\subset \mathbb{P}^2$ be defined as in Notation \ref{S}. Since $A\cup B\subset E$, we have $P\in \langle Y\rangle$. Since $r_Y(P) = 2d-2$
(see \cite{cs} or
\cite[Theorem 4.1]{lt}), we get $r_{X_{2,d}}(P)\leq 2d-2$. 

Assume $r_{X_{2,d}}(P) \le 2d-3$. Thus $\deg (A\cup B) \le
2d+1$. Take $L$ and
$F_2$ as in the statement of Lemma \ref{bbb}. Since $\deg (L\cap E) \le 2$, we have $\sharp (F_2) \ge 2$. Since $A$ is not reduced, we get $\sharp (A_{red}) \ge 3$, a contradiction.
\end{proof}

\begin{remark}\label{b2.1}
Assume $m \ge 2$ and that the scheme $A\subset \mathbb{P}^m$ is as in Proposition \ref{lb1}, i.e. it is contained in a smooth conic $E\subset \mathbb{P}^m$
and $\sharp (A)\le 2$. Set $Y:= \nu_{m,d}(E)$. In Proposition \ref{lb1} we proved that $r_Y(P)=2d-2$. Since one can use Sylvester's algorithm (see \cite[\S 3]{bgi}) to compute a set of points $S\subset C$ that evinces $r_Y(P)$, then one can use the same $S$ in order  to compute $r_{ X_{m,d}}(P)$, too.
\end{remark}

\begin{proposition}\label{lb2.4} Let $A\subset \mathbb{P}^2$ be a non-reduced 0-dimensional scheme of degree $4$ such that
$\mathcal {I}_A(2)$ is spanned. Moreover  suppose that a general conic $E\in |\mathcal{I}_A(2)|$ is not smooth: $E=L_1\cup
L_2$ with $L_1$ and $L_2$ lines and $L_1\ne L_2$. 
Assume $\sharp (A_{red})=3$. Fix any $P\in \langle \nu _{2,d}(A)
\rangle\setminus (\sigma_4^0(X_{2,d})\cup\sigma_3(X_{2,d}))$. If $d=3$, then $r_{X_{2,d}}(P) = 4$.
If $d\ge 4$, then $$r_{X_{2,d}}(P) = d+2.$$ 
\end{proposition}

\begin{proof} 
Let $A_O$ be the non-reduced connected component of $A$ and $O_1,O_2$ the reduced ones.
Since $\nu_{2,d}(A_O)$ is a tangent vector of $X_{2,d}$, $r_{X_{2,d}}(P')=d$ for all $P'\in \langle \nu_{2,d}(A_O)\rangle \setminus (A_O)_{red}$ (\cite[Theorem 32]{bgi}). Thus $r_{X_{2,d}}(P) \le d+2$. First assume $d\ge 4$. Using Lemma \ref{bbb} we easily get that  $r_{X_{2,d}}(P)  \ge d+2$ (see Proposition \ref{la1.3} for a similar case). If $d=3$, then we use \cite[Theorem 40]{bgi} or \cite[Remark 2.3.1]{kl} and get the inequality $r_{X_{2,3}}(P) \ge 4$.
\end{proof}

\begin{remark}\label{iii}
Observe that in this Section \ref{S} we are assuming that  $\mathcal {I}_A(2)$
is  spanned, this implies that a general $E\in |\mathcal{I}_A(2)|$ is reduced (Bertini's theorem) and $A$ is the complete intersection of two general elements of $|\mathcal{I}_A(2)|$.
\end{remark}

\begin{lemma}\label{b2.1.1}
Assume $\mathcal {I}_A(2)$ spanned (equivalently, assume $\deg (A\cap L)\le 2$ for every line $L\subset \mathbb {P}^2$)
and that a general $E\in \vert \mathcal {I}_A(2)\vert$ is not smooth. Then $A$ is connected, a complete intersection
and not curvilinear.
\end{lemma}

\begin{proof}
Since $E$ is not a double line (Remark \ref{iii}), $E$ has a unique singular point. Call it $Q$.
Since $\mathcal
{I}_A(2)$ is spanned,
$Q\in A_{red}$ (Bertini's theorem). Since $\vert \mathcal {I}_A(2)\vert$ is a pencil (and hence it is irreducible as an abstract
variety) and $A_{red}$ is finite, $Q\in \mbox{Sing}(F)$
for every $F\in \vert \mathcal {I}_A(2)\vert$. Take a general $F\in \vert \mathcal {I}_A(2)\vert \setminus \{E\}$. Both
$E$ and $F$ are reduced (Remark \ref{iii}). Since
no line is in the base locus of $\vert \mathcal {I}_A(2)\vert$, $E$ and $F$ have no common component. Since
$\deg (A)=4$, we get $A = E\cap F$ (scheme-theoretic intersection). Hence $A_{red} = \{P\}$ and $A$ is isomorphic to the scheme
defined around $(0,0)\in \mathbb {A}^2$ by two quadratic forms in two variables. Hence the Zariski tangent space of $A$ at $Q$
has dimension two.
\end{proof}

\begin{proposition}\label{lb2.1} Let $A\subset \mathbb{P}^2$ be a non-reduced zero dimensional scheme of degree $4$ such that $\mathcal {I}_A(2)$ is spanned
and a general $E\in |\mathcal{I}_A(2)|$ is not smooth. Write $E=L_1\cup L_2$ with $L_1$ and $L_2$ lines and $L_1\ne L_2$.

Assume $\sharp (A_{red})=1$.  Then$$r_{X_{2,d}}(P) = 2d-2$$for every $P\in \langle \nu _{2,d}(A) \rangle\setminus (\sigma_4^0(X_{2,d})\cup\sigma_3(X_{2,d}))$.\end{proposition}

\begin{proof}
Since $\deg (A\cap L_i) \le 2$ for all $i$, we have $A_{red} = L_1\cap L_2$.
Since $A$ is a Cartier divisor of $E$, we may apply Proposition
\ref{a1}. Thus $r_{\nu_{2,d}(E)}(P) = 2d-2$. Hence $r_{X_{2,d}}(P) \le 2d-2$. Thus it is sufficient to prove $r_{X_{2,d}}(P) \ge 2d-2$. Assume $r_{X_{2,d}}(P) \le 2d-3$. Hence $\deg (A\cup B) \le 2d+1$ for $B$ as in Notation \ref{S}. Thus we may apply Lemma \ref{bbb}. Since $A$ is connected, $F_2=\emptyset$, contradicting the inequality
$\sharp (F_2) \ge m-1=1$.
\end{proof}

Now we can prove that the inequality in Proposition \ref{a1} is actually an equality.

\begin{corollary}\label{ca1}
Fix an integer $d\ge 3$. Let $Y\subset \mathbb {P}^{2d}$ be a reduced and connected curve union of two smooth
degree $d$ curves $Y_1$, $Y_2$, each of them a rational normal curve in its linear span, with a unique
common point point, $Q$, and with $\langle Y\rangle = \mathbb {P}^{2d}$. Let $Z\subset Y$ be a
0-dimensional scheme such that $\deg (Z)=4$, $Z_{red}= \{Q\}$, $Z$ is a Cartier divisor of $Y$ and $\deg(Z\cap Y_i)\geq 2$ for $i=1,2$. Fix $P\in \langle Z\rangle$ such that $P\notin \langle Z'\rangle$ for any $Z'\subsetneq Z$. Then
$$r_Y(P) = 2d-2.$$
\end{corollary}

\begin{proof} The inequality $r_Y(P) \leq 2d-2$ is proved in Proposition \ref{a1}. 
\\
In the proof of Proposition \ref{la1.1} we showed that if $\deg (A\cap L_i)=3$ for one $i =1,2$ (i.e. if $\deg(Z\cap Y_i)\geq 2$ for one $i=1,2$), then $r_{X_{2,d}}(P)\geq 2d-2$. 
Since $Y\subseteq X_{2,d}$, we have $r_Y(P) \ge r_{X_{2,d}}(P) \ge 2d-2$, concluding
the proof.
\end{proof}

\subsection{Three dimensional case}\label{m3}

Here we assume that $m=3$ and that the degree $4$ non-reduced 0-dimensional scheme $A\subset \mathbb{P}^3$ introduced in Notation \ref{Z} is not contained in any plane of $\mathbb{P}^3$, i.e., $\dim (\langle A\rangle )=3$.

\begin{remark}\label{r12} 
If $A\subset \mathbb{P}^3$ is the first infinitesimal neighborhood $2Q$ of some point $Q\in \mathbb{P}^3$ then, if $Z\subset X_{3,d}$ is as in Notation \ref{Z}, the linear span $\langle Z\rangle$   
is actually  the tangent space $T_{\nu_{3,d}(Q)}X_{3,d}$ of $X_{3,d}$ at $\nu_{3,d}(Q)$. Therefore, by \cite[Theorem 32]{bgi}, we have $r_{X_{3,d}}(P) = d$, but also that $P\in \sigma_2(X_{3,d})$.
\end{remark}

\begin{proposition}\label{3a} \red{Assume $d\ge 3$.} Let $U_1,U_2\subset \mathbb{P}^3$ be two disjoint non-reduced 0-dimensional schemes of degree $2$ such that $A:= U_1\sqcup U_2$ spans $\mathbb{P}^3$. Then
$$r_{X_{3,d}}(P)=2d$$
for every $P\in \langle \nu _{3,d}(A) \rangle \setminus (\sigma_4^0(X_{3,d})\cup \sigma_3(X_{3,d}))$.
\end{proposition}

\begin{proof}
Proposition \ref{a2} gives $r_{X_{3,d}}(P) \le 2d$. Here we
will prove the reverse inequality and hence that $r_{X_{3,d}}(P) =2d$ for  \red{$d\ge 3$.} Assume $r_{X_{3,d}}(P) \le 2d-1$ and take $B\subset X_{3,d}$ such that
$\nu _{3,d}(B)$ evinces $r_{X_{3,d}}(P)$. By assumption we have $\deg (A\cup B) \le 2d+3$. Let $M \subset \mathbb {P}^3$ be a plane such that
$\deg (M\cap (A\cup B))$ is maximal. Consider the residual exact sequence (\ref{eqi1}) with $t=d$, $H = M$ and $E = A\cup B$.  Since
$h^1(\mathbb {P}^3,\mathcal {I}_{A\cup B}(d)) >0$ (Lemma \ref{v1}), we get that either $h^1(\mathcal {I}_{\mbox{Res}_M(A\cup B)}(d-1))>0$
or $h^1(M,\mathcal {I}_{M\cap (A\cup B}(d)) >0$. Since $h^1(\mathbb {P}^3,\mathcal {I}_{A\cup B}(d)) >0$ and $\deg (A\cup B) \le 2d+3 \le 3d+1$, $A\cup B$ is not
in linearly general position (\cite[Theorem 3.2]{eh}). Hence $\deg (M\cap (A\cup B)) \ge 4$.

\quad (a) In this step and in the next one we assume 
$h^1(M,\mathcal {I}_{(A\cup B)\cap M,M}(d)) >0$, i.e. $\langle \nu _{3,d}(A\cap M)\rangle \cap \langle \nu _{3,d}(B\cap M)\rangle \ne \emptyset$. In this step we also assume $\deg ((A\cup B)\cap M)\ge 2d+2$.
Since $A\cup B$ spans $\mathbb {P}^3$ and $\deg (A\cup B) \le 2d+3$, we get $\deg ((A\cup B)\cap M)=2d+2$ and
that
$\mbox{Res}_M(A\cup B)$ is a reduced point, say $Q$. Since $P\in \langle Z\rangle \cap \langle S\rangle$, to compute $r_{X_{3,d}}(P)$ we cannot
use a smaller number of variables (see \cite[Exercise 3.2.2.2]{l}, or \cite[Theorem 2.1]{bl},  for a generalization in the non-symmetric case). Thus $Q\in A_{red}\cap B_{red}$. Thus $\deg (A\cup B)\le \deg (A) +\deg (B\setminus
\{Q\}) \le 2d+2$. Since $\deg ((A\cup B)\cap M) \ge 2d+2$, we get $A\cup B \subset M$, a contradiction.

\quad (b) Here we assume $h^1(M,\mathcal {I}_{(A\cup B)\cap M,M}(d)) >0$
and $\deg ((A\cup B)\cap M)\le 2d+1$.  
Since $h^1(M,\mathcal {I}_{(A\cup B)\cap M,M}(d)) >0$ and $\deg ((A\cup B)\cap M) \le 2d+1$, there is a line
$L$ such that $\sharp ((A\cup B)\cap L) \ge d+2$ (\cite[Lemma 34]{bgi}). Since $\nu _{3,d}(B)$ is linearly independent, we have $\sharp (B\cap L) \le d+1$. Since $A$ spans $\mathbb {P}^3$ and $\deg (A)=4$, we have
$\deg (A\cap R)\le 2$ for every line $R\subset \mathbb {P}^3$. Therefore $\deg (A\cap L) \le 2$ and $d \le \deg (B\cap L) \le d+1$. Assume for the moment $\mbox{Res}_M(A)\cap \mbox{Res}_M(B)=\emptyset$. Since
$P\notin \langle \nu _{3,d}(M)\rangle$, Lemma \ref{c2} gives $h^1(\mathbb {P}^3,\mathcal {I}_{\mbox{Res}_M(A\cup
B)}(d-1))>0$. Hence $\deg (\mbox{Res}_M(A\cup B)) \ge d+1$. Since $\deg (A\cup B)\le 2d+2$ and $\deg ((A\cup B)\cap M) \ge
d+2$, we obtained a contradiction. Now assume $\mbox{Res}_M(A)\cap \mbox{Res}_M(B) \ne \emptyset$. Since
$\nu _{3,d}(B)$ is linearly independent, we must have $A\cap M \ne \emptyset$. Hence $M$ meets exactly one of the connected
components of $A$ and $B$ contains the support of the other connected component of $A$, say $(U_1)_{red} \in M$,
$(U_2)_{red}
\notin M$ and $(U_2)_{red}\in B$. Lemma \ref{c2} gives that $P$ is the linear span of $\langle \nu _{3,d}(A\cap M)\rangle$ and
the point $\nu _{3,d}((U_2)_{red})$. Hence $P\in \sigma _3(X_{3,d})$, a contradiction.

\quad (c) Now assume $h^1(M,\mathcal {I}_{(A\cup B)\cap M,M}(d))=0$.
Hence $h^1(\mathcal {I}_{\mbox{Res}_M(A\cup B)}(d)) > 0$ by the residual exact sequence. Since $\deg ((A\cup B)\cap M)\ge 4$, we have
$\deg (\mbox{Res}_M(A\cup B)) \le 2(d-1)+1$. Hence \cite[Lemma 34]{bgi},  gives the existence of a line $L\subset \mathbb {P}^3$ such that $\deg (L\cap \mbox{Res}_M(A\cup B)) \ge (d-1)+2$.
Since $\mbox{Res}_M(A\cup B)\subseteq A\cup B$, we get $\deg ((A\cup B)\cap L) \ge d+1$. Since $h^1(M,\mathcal {I}_{(A\cup B)\cap M,M}(d)) =0$, we have
$\deg ((A\cup B)\cap L) = d+1$. Let $N \subset \mathbb {P}^3$ be a general plane containing
$L$. Since $A\cup B$ is curvilinear and $(A\cup B)_{red}$ is finite, we have $N\cap (A\cup B) = L\cap (A\cup B)$ (as schemes). Hence $h^1(N,\mathcal {I}_{N\cap (A\cup B)}(d))=0$. The residual exact
sequence of $N$ gives $h^1(\mathcal {I}_{\mbox{Res}_N(A\cup B)}(d-1)) >0$. Since $\deg (\mbox{Res}_N(A\cup B)) \le 2d+3-d-1\le 2(d-1)+1$, there is
a line $T\subset \mathbb {P}^3$ such that $\deg (T\cap \mbox{Res}_N(A\cup B))\ge d+1$. Since $B$ is reduced and $L\subset N$, we have $T\nsubseteq N$ and in particular
$T\ne L$. Hence $\sharp (B\cap T \setminus B\cap T\cap L) \ge d-2$. Fix any $o\in B\cap T \setminus (T\cap B\cap T\cap L)$. Let $N_o$ be the plane spanned by $L$ and $o$. If $\deg (N_o\cap (A\cup B)) \ge 2d+2$, part (a) gives a contradiction.
Hence we may assume $\deg (N_o\cap (A\cup B)) \le 2d+1$. First assume $h^1(N_o,\mathcal {I}_{N_o\cap (A\cup B)}(d)) =0$. The residual
sequence of $N_o$ gives $h^1(\mathcal {I}_{\mbox{Res}_{N_o}(A\cup B)}(d-1)) >0$. Hence there is a line $T_o$ such that $\deg (T_o\cap \mbox{Res}_{N_o}(A\cup B) )\ge d+1$.We may assume $\deg (T_o\cap (A\cup B))=d+1$ (e.g. by parts (a) and (b)). Since $B$ is reduced, and $o\in B\cap N_o$, we
have $o\notin T_o$. Since $N_o\supset L$, we have $T_o\ne L$. We get $\deg (A\cup B) \ge 3(d+1)-2$, a contradiction. Now assume $h^1(N_o,\mathcal {I}_{N_o\cap (A\cup B)}(d)) > 0$. Since $\deg (N_o\cap (A\cup B) \le 2d+1$,
there is a line $D_o \subset N_o$ such that $\deg (D_o\cap (A\cup B)) \ge d+2$. Hence $D_o\ne L$ and $D_o\ne T$. The lines $L$, $T$ and $D_o$ give $\deg (A\cup B) > 2d+3$, a contradiction.\end{proof}

\begin{remark}\label{3aRemark}
Assume, for $m > 2$, that the 0-dimensional scheme $A\subset \mathbb{P}^m$ of Notation \ref{Z} has two connected components, $A_1, A_2\subset \mathbb{P}^m$, both of degree $2$ and that the lines $L_i:= \langle A_i\rangle$, $i=1, 2$,
are disjoint. Thus $\dim (\langle L_1\cup L_2\rangle )=3$. Set $Y_i:= \nu_{m,d}(L_i)$, $i=1,2$, and $Y:= Y_1\cup Y_2$. Notice that $Y_1\cap Y_2=\emptyset$. Now let $Z\subset X_{m,d}$ be defined as in Notation \ref{Z} as a scheme that evinces the $X_{m,d}$-border rank of a point $P\in \langle Z \rangle\setminus (\sigma_4^0(X_{m,d})\cup \sigma_3(X_{m,d}))$.  By \cite[Proposition 3.1]{ls},  or \cite[Subsection 3.2]{lt},  $r(X_{m,d})(P) = r_{Y_1\cup Y_2}(P)$.
We proved in Proposition \ref{a2} that $r_{X_{m,d}}(P) = 2d$ and that it may be evinced by a set $S\subset Y$ such that $\sharp (S\cap Y_i)=d$, $i=1,2$. The set $S$ may be found in the following way
(here we just translate the proof of Proposition \ref{a2}): 

\quad Step 1. Set $P_2:= \langle \{P\}\cup Y_1\rangle \cap \langle Y_2\rangle$ and $P_1:= \langle \{P\}\cup Y_2\rangle \cap \langle Y_1\rangle$.

\quad Step 2. Find $S_i\subset Y_i$ evincing the $Y_i$-rank of $P_i$ (e.g. use Sylvester's algorithm \cite{cs}, \cite{bgi}, \cite{cglm} and \cite{bcmt}). 

\quad Step 3. Set $S:= S_1\cup S_2$.
\end{remark}

\begin{proposition}\label{3bp}
\red{Assume $d\ge 3$.} Let $A\subset \mathbb{P}^3$ be a degree 4 curvilinear 0-dimensional scheme with support at only one point and such that $\langle A \rangle =\mathbb{P}^3$. Then
$$r_{X_{3,d}}(P)=3d-2$$ 
for all $P\in \langle \nu _{3,d}(A) \rangle \setminus (\sigma_4^0(X_{m,d})\cup \sigma_3(X_{3,d}))$.
\end{proposition}

\begin{proof}
Since $A$ spans $\mathbb {P}^3$, it is projectively
equivalent to a connected degree $4$ divisor of a smooth rational normal curve $Y$ of $\mathbb {P}^3$. Thus $r_{X_{3,d}}(P) \le r_{\nu _{3,d}(Y)}(P) =3d-2$ (\cite{cs}). 
In order to obtain a contradiction we assume $r_{X_{3,d}} (P)\le 3d-3$.

Take $B \subset \mathbb {P}^3$ such that $\nu _{3,d}(B)$ evinces $r_{X_{3,d}}(P)$. We have $\deg (A\cup B)=4+r_{X_{3,d}} (P) -\deg(A\cap B)
\le 3d+1$. Lemma \ref{v1} gives $h^1(\mathbb {P}^m,\mathcal {I}_{A\cup B}(d)) >0$. Hence $A\cup B$ is not in linearly general
position (see \cite[Theorem 3.2]{eh}). Thus there is a plane $M\subset \mathbb {P}^3$ such
that $\deg (M\cap (A\cup B))\ge 4$. Among all such planes we take one, say $M_1$, such that $x_1:= \deg (M_1\cap (A\cup B))$ is maximal. Set
$E_1:= A\cup B$ and $E_2:= \mbox{Res}_{M_1}(E_1)$. Notice that $\deg (E_2)=\deg (E_1)-x_1$. Define
inductively the planes $M_i \subset \mathbb {P}^3$, $i \ge 2$, the schemes $E_{i+1}$, $i\ge 2$, and the integers $x_i$,
$i\ge 2$, by the condition that
$M_i$ is one of the planes such that $x_i:= \deg (M_i\cap E_i)$ is maximal and then set $E_{i+1}:= \mbox{Res}_{M_i}(E_i)$. We
have
$E_{i+1}\subseteq E_i$ (with strict inclusion if $E_i \ne \emptyset$) for all $i\ge 1$ and $E_i=\emptyset$ for all $i \gg 0$.
For all integers
$t$ and
$i
\ge 1$ there is the residual exact sequence
\begin{equation}\label{eqd1}
0 \to \mathcal {I}_{E_{i+1}}(t-1) \to \mathcal {I}_{E_i}(t) \to \mathcal {I}_{E_i\cap M_i,M_i}(t) \to 0
\end{equation}
Let $z$ be the minimal integer $i$ such that $1 \le i\le d+1$ and $h^1(M_i,\mathcal {I}_{M_i\cap E_i}(d+1-i))>0$.
Use at most $d+1$ times the exact sequences (\ref{eqd1}) to prove the existence of such an integer $z$. We now study the different possibilities that we have for the integer $z$ just defined.

 \quad (a) Here we assume $z=1$. Since $\nu _{3,d}(B)$ is linearly independent and $h^1(M_1,\mathcal {I}_{(A\cup B)\cap
M_1}(d))>0$,
we have $A_{red}\in M_1$. Since $B$ is reduced, we get $\mbox{Res}_{M_1}(A)\cap
\mbox{Res}_{M_1}(B) =\emptyset$. Since $P\notin \langle \nu _{3,d}(M_1)\rangle$,
Lemma
\ref{c2} gives
$h^1(\mathbb {P}^3,\mathcal {I}_{E_2}(d-1)) >0$. Hence $x_2 \ge d+1$. Since by hypothesis \red{$d \ge 3$}, $x_2\le x_1$ and $x_1+x_2
\le 3d+1$, we have $x_2\le 2d-1$. Hence there is a line $R\subset \mathbb {P}^3$ such that $\deg (E_2\cap R)\ge d+1$ (\cite[Lemma 34]{bgi}).
Hence $x_1\le 2d$. Since $h^1(M_1,\mathcal {I}_{(A\cup B)\cap
M_1}(d))>0$, \cite[Lemma 34]{bgi}, gives the existence of a line $L\subset M_1$ such that
$\deg (L\cap (A\cup B))\ge d+2$. Since $\nu _{3,d}(B)$ is linearly independent, we get
$A\cap L \ne \emptyset$. Since $B$ is reduced, $\deg (A\cap T)\le 2$ for any line $T\subset \mathbb {P}^3$ and $\deg (R\cap
\mbox{Res}_{M_1}(A\cup B))\ge 3$, we have $R\nsubseteq M_1$ and in particular $R\ne L$. First assume $R\cap L \ne \emptyset$.
Since $\deg ((A\cup B)\cap \langle L\cup R\rangle )\ge 2d+1$, we have $x_1\ge 2d+1$. Hence $x_2\le d$, a contradiction.
Now assume $L\cap R = \emptyset$. In particular we have $A_{red}\notin R$. Hence $\sharp (R\cap B)\ge d+1$.
Since $\nu _{3,d}(B)$ is linearly independent, we have $\sharp (R\cap B)=d+1$. Fix any $Q\in R\cap B$. Let
$H\subset \mathbb {P}^3$ be the plane spanned by $L$ and by $Q$. Since $A_{red}\in L$ and $B$ is reduced, we have
$\mbox{Res}_H(A)\cap \mbox{Res}_H(B)=\emptyset$. Lemma \ref{c2} gives $h^1(\mathcal {I}_{\mbox{Res}_H(A\cup B)}(d-1))>0$.
Since $\deg ((A\cup B)\cap H)\ge d+3$, we have $\deg (\mbox{Res}_H(A\cup B))\le 2d-2 \le 2(d-1)+1$. Hence
there is a line $R'\subset \mathbb {P}^3$ such that $\deg (R'\cap (\mbox{Res}_H(A\cup B)) \ge d+1$. Since
$L\subset H$ and $B$ is reduced, we have $R'\ne L$. Since $\deg ((A\cup B)\cap R)=d+1$ and $H$ contains
one of the points of $R$, we have $R'\ne R$. If $R' \cap L\ne \emptyset$, using the plane
$\langle L\cup R'\rangle$ we get
$x_1\ge 2d+1$ and $x_2\ge d+1$, a contradiction. If $R'\cap L =\emptyset$, then $\deg (A\cup B)\ge d+2+2d+1$, a contradiction.

\quad (b) From now on we assume $z>1$.
Since $h^1(M_z,\mathcal {I}_{M_z\cap E_z}(d+1-z))>0$, we have
$x_z:= \deg (M_z\cap E_z) \ge d+3-z$. Since the function $z\mapsto x_z$ is non-increasing,
we get $x_i\ge d+3-z$ for all $i\in \{1,\dots ,z+1\}$. Since $\deg (A\cup B) \ge z(d+3-z)$, we get
$3d+1 \ge z(d+3-z)$. Hence either $z\in \{2,3\}$ or $z\ge d$  \red{(this statement is trivially true if $d=3,4$).}

\quad (c) Assume $z=d$.
The condition
$h^1(\mathcal {I}_{M_d\cap E_d}(1)) >0$ says that either $M_d\cap E_d$ contains a scheme of length $\ge 3$ contained in a line
$R$ or $x_d\ge 4$. If $x_d \ge 4$, then we get $x_1+\cdots +x_d \ge 4d$, that is a contradiction. Hence we may assume
$x_1=4$,
$x_i =3$ for
$2\le i
\le d$ and $x_{d+1} =0$.  Since $x_2=3$, the maximality of the integer $x_2$ gives that $E_2$ is in linearly general position.
Since
$\deg (E_2) = \deg (E_1)-x_4 \le 3(d-1)+1$ and $E_2$ is in linearly general position, then $h^1(\mathcal {I}_{E_2}(d-1))=0$. Since $z>1$. $h^1(M_1,\mathcal {I}_{E_1\cap M_1}(d))=0$. Hence
(\ref{eqd1}) with $i=1$ and $t=d$ gives a contradiction.

\quad (d) Assume $z=d+1$.
The condition $h^1(M_i,\mathcal {I}_{M_z\cap E_z})>0$ only says $x_{d+1} \ge 2$. Taking the first integer $y\le d$ such that
$x_y \le 3$ and $E_y$ is not collinear, we get a contradiction as above.

\quad (e) Assume $z=2$.
Since $3d+1 \ge x_1+x_2 \ge 2x_2$, we get
$x_2 \le 2(d-1)+1$. By \cite[Lemma 34]{bgi},  there is a line $R\subset M_2$ such that $\deg (\mbox{Res}_{M_1}(A\cup B) \cap R)) \ge d+1$. Hence $x_2 \ge d+1$. Since $x_2\ge d+1$, we have $x_1\le 2d$. Since
$\mbox{Res}_{M_1}(B)\cap R\ne \emptyset$ and $B$ is reduced, we have $R\nsubseteq M_1$. Since $z>1$ and $R\subset M_1$, we
get $\deg ((A\cup B)\cap R)=d+1$. Let $H$ be a plane containing $R$ and such that
$e_1:= \deg ((A\cup B)\cap H)$ is maximal. Since $A\cup B$ spans $\mathbb {P}^3$ we have $e_1\ge d+2$. First assume $h^1(H,\mathcal {I}_{H\cap (A\cup B)}(d)) > 0$. Since
$e_1 \le z_1 \le 2d$, \cite[Lemma 34]{bgi}, gives the existence of a line $L\subset H$ such that $\deg (L\cap (A\cup B)) \ge d+2$. Since $\deg (R\cap (A\cup B))=d+1$, we have $L\ne R$. Since
the scheme $L\cup R$ has degree $1$ and $L\cup R\subset H$, we get $e_1 \ge (d+2)+(d+1)-1$. Hence $z_1>2d$, a contradiction. Now assume $h^1(H,\mathcal {I}_{H\cap (A\cup B)}(d))=0$.
The residual exact sequence (\ref{eqi1}) gives $h^1(\mathcal {I}_{\mbox{Res}_H(A\cup B)}(d-1)) > 0$. We have $\deg (\mbox{Res}_H(A\cup B)) =\deg (A\cup B)-e_1 \le 2d-1$. Hence
there is a line $D\subset \mathbb {P}^3$ such that $\deg (\mbox{Res}_H(A\cup B)\cap D) \ge d+1$. Since $B$ is reduced and $\mbox{Res}_H( B)\cap D\ne \emptyset$, we have $D\nsubseteq H$
and in particular $D \ne R$. For any $o\in \mbox{Res}_H( B)\cap D$ set $N_o:= \langle R\cup \{o\}\rangle$. Since $R\subset H$ and $o\notin H$, we have $o\notin R$. Hence $N_o$ is a plane.
We have $\deg (N_o\cap (A\cup B)) \ge d+2$. If $h^1(N_o,\mathcal {I}_{N_o\cap (A\cup B)}(d)) > 0$, then as above we get a contradiction. Hence we may assume $h^1(N_o,\mathcal {I}_{N_o\cap (A\cup B)}(d))=0$.
A residual exact sequence gives $h^1(\mathcal {I}_{\mbox{Res}_{N_o}(A\cup B)}(d-1)) > 0$. Since  $\deg (\mbox{Res}_{N_o}(A\cup B)) \le 2d-1$, \cite[Lemma 34]{bgi},  gives
the existence of a line $L_o\subset N_o$ such that $\deg (\mbox{Res}_{N_o}(A\cup B)\cap L_o) \ge d+1$. We have $N_o\cap N_{o'} = R$ for all $o'\notin N_{o}$.
Set $\alpha := \sharp (B\cap R)$. Since $\deg (A\cap T)\le 2$ for all lines $T$, we have $\alpha \ge d-1$. We get $\deg (A\cup B) \ge \deg ((A\cup B)\cap R) +\alpha d
\ge d+1+(d-1)d$, a contradiction.

\quad (f) Assume $z=3$.
Since $x_1\ge x_2\ge x_3 \ge d$, we get $x_2 =x_3=d$, $x_1\le d+1$ and the
existence of a line
$R
\subset M_3$ such that
$E_3\cap M_3 \subset R$. Since $M_3$ is a plane for which $\deg (E_3\cap M_3)$ is maximal, while there is a pencil of planes
containing $R$, we have $E_3 \subset M_3$ and $E_4=\emptyset$. Now instead of
$M_2$ we take a plane
$M'_2$ containing
$R$ and at least another point of
$B\setminus B\cap M_1$. Since $\deg (M'_2\cap E_2) \ge d+1$, we have $x_2\ge d+1$, a contradiction.

Therefore we may conclude that $r_{X_{3,d}}(P)=3d-2$.
\end{proof}

\begin{remark}\label{3b}
Fix $P\in \sigma _4(X_{m,d})\setminus \sigma _3(X_{m,d})$, $m \ge 3$ and \red{$d \ge 3$}, for which $A\subset \mathbb{P}^m$ and hence $Z\subset X_{m,d}$ are as in Proposition \ref{3bp}. Here we want to describe
and produce algorithmically several sets of points $S\subset X_{m,d}$ evincing $r_{X_{m,d}}(P)$.
\\
Fix any $3$-dimensional linear subspace $M$ of $\mathbb {P}^m$
containing $A$ and any smooth rational normal curve $T$ of $M$ such that $A\subset T$. Set $Y:= \nu_{m,d}(T)$. Thus $Y$ is a degree $3d$ rational
normal curve in its linear span. Since $Z\subset Y$, we have $P\in \langle Y\rangle$. Since $\deg (Z)=4$ and $Z$ is contained in a rational
normal curve, we have $r_Y(P) = 3d-2$ (see \cite{cs} or \cite[Theorem 4.1]{lt}). Hence $r_Y(P) = r_{X_{m,d}}(P)$. Hence any $S\subset Y$ evincing
$r_Y(P)$ evinces $r_{X_{m,d}}(P)$. Sylvester's algorithm produces one such set $S$ (see \cite{cs}, \cite{bgi}, \cite{cglm}, \cite{bcmt}).
\end{remark}

\begin{lemma}\label{r15}
Fix $O\in \mathbb {P}^3$. Let $A$ be a 0-dimensional scheme of degree $4$ such
that $\deg (A)=4$, $A_{red} = O$, $\langle A\rangle =\mathbb {P}^3$ and $A$ is not curvilinear. Then $A$ is the first infinitesimal neighborhood
of $O$ in $\mathbb {P}^3$ and $\langle \nu _{3,d}(A)\rangle \subset \sigma _2(X_{3,d})$.
\end{lemma}

\begin{proof}
 Since $A$ is not curvilinear and $\deg (A) = \dim (\langle A\rangle )+1$, $A$ is not as in case
III of \cite[Theorem 1.3]{eh}. Hence \cite[Theorem 1.3]{eh}, gives that $A$ is the first infinitesimal neighborhood of $O$ in $\mathbb {P}^3$.
Since $A$ is the first infinitesimal neighborhood of $O$ in $\mathbb {P}^3$, every point of  $\langle \nu _{3,d}(A)\rangle$ is contained in the tangent developable of $X_{3,d}$ and hence in $\sigma _2(X_{3,d})$.
\end{proof}

\begin{proposition}\label{p15} 
Let $A_1\subset \mathbb{P}^3$ be a degree 3 connected 0-dimensional scheme contained in a smooth conic. Fix $O\in \mathbb {P}^3\setminus \langle A_1\rangle$
and set $A:=A_1\sqcup \{O\}$. Then 
$$r_{X_{3,d}}(P)=2d$$
for every $P\in\langle \nu _{3,d}(A)\rangle \setminus (\sigma_4^0(X_{3,d})\cup \sigma_3(X_{3,d}))$.
\end{proposition}

\begin{proof}
Assume $r_{X_{3,d}}(P) \le 2d-1$ and take $B\subset \mathbb {P}^3$ such that $\nu _{3,d}(B)$ evinces $r_{X_{3,d}}(P)$. Set $Q:= (A_1)_{red}$. We may repeat verbatim the proof of Proposition \ref{3a}, except the last
part of step (b): the case $\mbox{Res}_M(A)\cap \mbox{Res}_M(B) \ne \emptyset$, i.e. $B\setminus B\cap M$ contains at least one of the points $O, Q$. We may also assume $\deg (M\cap (A\cup B)) \le 2d+1$
and $h^1(\mathcal {I}_{\mbox{Res}_M(A\cup B)}(d-1))=0$.  Since $\deg (B\cap L) \le d+1$, we have $Q\in L$. Hence $Q\in M$ and $Q\notin \mbox{Res}_M(B)$ even if $Q\in B$. 
Hence we may assume $O\notin M$ and $O\in B$. Set $B_1:= B\setminus \{O\}$. Notice that $(B_1\setminus B_1\cap L)\cap A_1)=\emptyset$.
we have $h^1(M,\mathcal {I}_{\mbox{Res}_L(\mbox{Res}_L((A\cup B)\cap M)}(d-1))=0$. Hence $h^1(M,\mathcal {I}_{M\cap (A\cup B)}(d)) = h^1(L,\mathcal {I}_{L\cap (A\cup B)}(d))$. Since $h^1(\mathcal {I}_{\mbox{Res}_M(A\cup B)}(d-1))=0$,
we get $h^1(\mathcal {I}_{A\cup B}(d)) =h^1(L,\mathcal {I}_{L\cap (A\cup B)}(d))$. Since $L\cap A =L\cap A_1$ and $B\cap L = B\cap L_1$, we also get
$h^1(\mathcal {I}_{A_1\cup B_1}(d)) =h^1(L,\mathcal {I}_{L\cap (A_1\cup B_1)}(d))$. Since  $(B_1\setminus B_1\cap L)\cap A_1)=\emptyset$, as in the first proof of lemma
\ref{c2} we get $\langle \nu _{3,d}(A_1)\rangle \cap \langle \nu _d(B_1)\rangle = \langle \nu _{3,d}(A_1\cap L)\rangle \cap \langle \nu _d(B_1\cap L)\rangle$. Since $h^1(\mathcal {I}_{A\cup B}(d)) =h^1(\mathcal {I}_{A_1\cup B_1}(d))$,
Grassmann's formula gives that $\langle \nu _{3,d}(A)\rangle \cap \langle \nu _d(B)\rangle $ is spanned by its subspaces $\langle \nu _{3,d}(A_1)\rangle \cap \langle \nu _d(B_1)\rangle $ and $\nu _{3,d}(O)$.
Since $\deg (A_1\cap L) \le 2$, we get $P\in \sigma _3(X_{3,d})$, a contradiction.\end{proof}

\section{Proof of the main theorem}

\begin{remark}\label{theR} Fix a 0-dimensional scheme $A\subset \mathbb {P}^m$ of degree $4$ and set $s:= \dim (\langle A\rangle )$. We have $s\le \min \{m,3\}$. Set $W:= \nu _{m,d}(\langle A\rangle )$.
Hence $W$ is projectively equivalent to $ X_{s,d}$.
Fix any $P\in \langle \nu _{m,d}(A) \rangle $. We have $r_{X_{m,d}}(P) = r_W(P)$ (\cite[Proposition 3.1]{ls}) and any $S\subset X_{m,d}$ evincing $r_{X_{m,d}}(P) $
is contained in $W$ (see \cite[Exercise 3.2.2.2]{l}, or \cite[Theorem 2.1]{bl}, for a generalization to the non-symmetric case). Hence not only to prove Theorem \ref{theorem} it is sufficient to do the case $m\le 3$, but any way of producing a set evincing $r_{X_{m,d}}(P) $ must work (implicitly or explicitly)
inside $\langle A\rangle$.
\end{remark}

 \vspace{0.3cm}

\qquad {\emph{ Proof of Theorem \ref{theorem}}.}
We want to classify the $X$-rank of a point $P\in\sigma_4(X)\setminus \sigma_3(X)$ where $X$ is the Veronese embedding of $\mathbb{P}^m$ into $\mathbb{P}^{n}$ with $n={m+d\choose d}$. 

Now $\sigma_4(X)\setminus\sigma_3(X)$ is the union of 
$$\sigma_4^0(X)\setminus\sigma_3(X)=\{P\in \sigma_4(X)\, | \, r_X(P)=4  \}$$ 
(the set $\sigma_4^0(X)$ is defined in (\ref{sigma0})) and $$\sigma_4(X)\setminus(\sigma_4^0(X)\cup \sigma_3(X))=\{P\in \sigma_4(X)\, | \, r_X(P)>4  \}.$$ 
We have obviously to study only the $X$-rank of points $P\in \sigma_4(X)\setminus(\sigma_4^0(X)\cup \sigma_3(X))$. 
In order to do that, as already showed in Section \ref{S1}, we have to study the $X$-rank of points belonging to the span of a  0-dimensional non-reduced smoothable and Gorenstein sub-scheme $Z\subset X$ of degree $4$ evincing the $X$-border rank of such a point  $P\in\sigma_4(X)\setminus(\sigma_3(X)\cup\sigma_4^0(X))$ (as in  Notation \ref{Z}).
 
By Remark \ref{theR} we may restrict our attention to the case $m\leq 3$. Therefore we study separately the cases $m=1,2, 3$ (we will do them in the following items (I), (II) and (III) respectively).

\textbf{\quad (I) Assume $\mathbf{m=1}$.} In this case $Z=\nu_{m,d}(A)$ for $A$ contained in a line $L\subset \mathbb{P}^m$, hence $r_{X_{m,d}}(P)=r_{\nu_{m,d}(L)}(P)=d-2$ (for \cite{cs}, \cite{bgi}, \cite{cglm}, \cite{bcmt} or \cite[Theorem 4.1]{lt}).

\red{This case (I) proves that the set $\sigma_{4,d-2}(X_{m,d})$ has to appear in all cases of the statement of the Theorem where $d-2\geq 4$ (ie. cases (a), (e) and (i) of Theorem \ref{theorem}; moreover, for the case (a) there are no other cases to consider).}

\textbf{\quad (II) Assume $\mathbf{m=2}$.} The scheme $A$ now is a 0-dimensional scheme of degree $4$ that  is contained in a plane but not in a line (otherwise we are again in case (I)), hence it can intersect at least one line in degree $3$ or it does not exist any line that intersects $A$ in degree $3$.

\quad (II1) If $\deg(A\cap L)=3$ for at least one line $L\subset \mathbb{P}^m$ then we distinguish the following cases:

\quad (II1.1) If $Res_L(A)\notin L$ then $r_{X_{m,d}}(P)=d$ by Proposition \ref{la1}.

\quad (II1.2) If $Res_L(A)\in L$ then we study the cardinality of the support of the scheme $A$.

\quad (II1.2.1) If $\sharp(Supp(A))=1$, then $r_{X_{m,d}}(P)=2d-2$ by Proposition \ref{la1.1}
(case $A$ not curvilinear)
and Lemma \ref{m3lemma} (case $A$ curvilinear).

\quad (II1.2.2) If $\sharp(Supp(A))=2$, then either $A$ is the union of two non-reduced 0-dimensional schemes both of degree $2$ or $A$ is the union of a simple point $O$ and a first infinitesimal neighborhood of another point $Q\in\mathbb{P}^2$. In the first case $r_{X_{m,d}}(P)=2d-2$ by Lemma \ref{m1}, in the second case we have that $P\in \langle O, T_{\nu_{2,d}(Q)}X \rangle$, but since $T_QX\subset \sigma_2(X)$, then  $P\in \sigma_3(X_{m,d})$.

\quad (II1.2.3) If $\sharp(Supp(A))=3$, then $r_{X_{m,d}}(P)=d+2$ by Proposition \ref{la1.3} \red{(if $d\geq 4$)}.

\quad (II2) Now assume that $\deg(A\cup L)<3$ for all lines $L$'s contained in $\mathbb{P}^m$ and $\dim(\langle A\rangle)=2$.

\quad (II2.1) If the generic conic through $A$ is smooth, then, by Proposition \ref{lb1}, $r_{X_{m,d}}(P)=2d-2$, except if $\sharp (A_{red}) =3$; in the latter case $r_{X_{m,d}}(P)=d+2$ by Proposition \ref{lb2.4} \red{(if $d\geq 4$)}.

\quad (II2.2) If the generic conic through $A$ is not smooth, then $r_{X_{m,d}}(P)=2d-2$ by Lemma \ref{b2.1.1} and Proposition \ref{lb2.1}. 

\red{This case (II) proves that the subsets $\sigma_{4,d}(X_{m,d})$ and $\sigma_{4,d+2}(X_{m,d})$
%$\sigma_{4,2d-1}(X_{m,d})$ 
have to appear in all cases of the statement of the Theorem where $d\geq 4$ and $m\geq 2$; the subset $\sigma_{4,2d-2}(X_{m,d})$ has to appear in all cases where $2d-2\geq 4$ and $m\geq 2$.
This completes the proofs of the cases (c), (d) and (e) of the statement of Theorem \ref{theorem}. Observe that the case (b) is covered by \cite[Theorem 40]{bgi}.}

\textbf{\quad (III) Assume $\mathbf{m=3}$.}

\quad (III1) If $\sharp(Supp(A))=1$ we may assume that $A$ is not the first infinitesimal neighborhood of a point $Q\in \mathbb{P}^3$, otherwise $P\in \sigma_2(\nu_{m,d}(\langle A \rangle))\subset \sigma_2(X_{m,d})$. Remark
\ref{r15} and Proposition \ref{3bp} give $r_{X{m,d}}(P)=3d-2$.

\quad (III2) If $\sharp (Supp (A))=2$ we may have the following cases.

\quad (III2.1) The scheme $A$ is the union of a simple point $O$ and a 0-dimensional scheme $A'$ of degree $3$ supported on a point $Q\subset \mathbb{P}^3$ such that $\dim(\langle A' \rangle)=2$ and $\langle \nu_{3,d}( A') \rangle \subset T_{\nu_{3,d}(Q)}X$. Therefore $P\in \langle O, T_{\nu_{3,d}(Q)}X \rangle \subset \langle O , \sigma_2(X) \rangle \subset \sigma_3(X)$.

\quad (III2.2) The scheme $A$ is the union of two non-reduced 0-dimensional schemes both of degree $2$. Since $\langle A \rangle= \mathbb{P}^3$ we are in the case of Proposition \ref{3a} where we get that $r_{X_{m,d}}(P)=2d$.

\quad (III2.3) The scheme $A$ is the union of a simple point and of a degree $3$ curvilinear 0-dimensional scheme supported on one point. Proposition \ref{p15} gives us that $r_{X_{m,d}}(P)=2d$.

\quad (III3) If $\sharp (Supp (A))=3$ then $A$ can only be the union of two simple points and a degree $2$ non-reduced scheme. By Remark \ref{2a1.3} we have that $r_{X_{m,d}}(P)=d+2$.

\red{This proves that the sets $\sigma_{4,3d-2}(X_{m,d})$ and  $\sigma_{4,2d}(X_{m,d})$ have to appear in all cases of the statement of the Theorem where $d\geq 3$ and $m\geq 3$. This completes the proof of the cases (f), (g), (h) and (i) of the statement of Theorem \ref{theorem}.}

\providecommand{\bysame}{\leavevmode\hbox to3em{\hrulefill}\thinspace}

\end{document}